\documentclass[a4paper]{amsart}

\usepackage{amssymb, enumitem}
\usepackage[all]{xy}
\usepackage{tikz} 
\usepackage{hyperref}
\usepackage{aliascnt}
\usepackage{verbatim}

\usepackage[english]{babel}

\usepackage{enumitem}
\usepackage{mathtools}
\usepackage{todonotes}
\usepackage{tikz-cd}

\def\today{\number\day\space\ifcase\month\or   January\or February\or
   March\or April\or May\or June\or   July\or August\or September\or
   October\or November\or December\fi\   \number\year}

\theoremstyle{plain}
\newtheorem{lma}{Lemma}[section]

\newaliascnt{thmCt}{lma}
\newtheorem{thm}[thmCt]{Theorem}
\aliascntresetthe{thmCt}

\newaliascnt{corCt}{lma}
\newtheorem{cor}[corCt]{Corollary}
\aliascntresetthe{corCt}

\newaliascnt{propCt}{lma}
\newtheorem{prop}[propCt]{Proposition}
\aliascntresetthe{propCt}

\newtheorem*{thm*}{Theorem}
\newtheorem*{cor*}{Corollary}
\newtheorem*{prop*}{Proposition}

\theoremstyle{plain}
\newtheorem{thmintro}{Theorem} 

\newaliascnt{introcorCt}{thmintro}

\aliascntresetthe{introcorCt}

\newaliascnt{egintroCt}{thmintro}
\newtheorem{egintro}[egintroCt]{Example}

\theoremstyle{definition}

\newaliascnt{introquesCt}{thmintro}

\aliascntresetthe{introquesCt}

\newaliascnt{pbmintroCt}{thmintro}
\newtheorem{pbmintro}[pbmintroCt]{Problem}
\aliascntresetthe{introquesCt}

\newaliascnt{dfnintroCt}{thmintro}
\newtheorem{dfnintro}[dfnintroCt]{Definition}
\aliascntresetthe{introquesCt}

\newaliascnt{pgrCt}{lma}

\aliascntresetthe{pgrCt}

\newaliascnt{dfCt}{lma}
\newtheorem{df}[dfCt]{Definition}
\aliascntresetthe{dfCt}

\newaliascnt{remCt}{lma}
\newtheorem{rem}[remCt]{Remark}
\aliascntresetthe{remCt}

\newaliascnt{remsCt}{lma}

\aliascntresetthe{remsCt}

\newaliascnt{egCt}{lma}
\newtheorem{eg}[egCt]{Example}
\aliascntresetthe{egCt}

\newaliascnt{egsCt}{lma}

\aliascntresetthe{egsCt}

\newaliascnt{qstCt}{lma}

\aliascntresetthe{qstCt}

\newaliascnt{pbmCt}{lma}

\aliascntresetthe{pbmCt}

\newaliascnt{notaCt}{lma}

\aliascntresetthe{notaCt}

\theoremstyle{theorem}
\newtheorem{clm}{Claim}[thmCt]

\newcommand{\beq}{\begin{equation}}
\newcommand{\eeq}{\end{equation}}
\newcommand{\beqa}{\begin{eqnarray*}}
\newcommand{\eeqa}{\end{eqnarray*}}
\newcommand{\bal}{\begin{align*}}
\newcommand{\eal}{\end{align*}}
\newcommand{\bi}{\begin{itemize}}
\newcommand{\ei}{\end{itemize}}
\newcommand{\be}{\begin{enumerate}}
\newcommand{\ee}{\end{enumerate}}

\newcommand{\acton}{\curvearrowright}

\newcommand{\ep}{\varepsilon}

\newcommand{\Q}{{\mathbb{Q}}}

\newcommand{\Z}{{\mathbb{Z}}}

\newcommand{\N}{{\mathbb{N}}}

\newcommand{\K}{{\mathcal{K}}}

\newcommand{\diag}{{\mathrm{diag}}}
\newcommand{\supp}{{\mathrm{supp}}}

\newcommand{\Aut}{{\mathrm{Aut}}}

\newcommand{\Prob}{{\mathrm{Prob}}}

\newcommand{\dd}{^{**}}
\newcommand{\dfin}{^{**}_{\mathrm{fin}}}

\numberwithin{equation}{section}

\newcommand{\I}{\infty}

\title[]{Tracially amenable actions and purely infinite crossed products}

\thanks{All authors were supported by the Deutsche Forschungsgemeinschaft (DFG, German Research Foundation) under Germany's Excellence Strategy EXC 2044 –390685587, Mathematics M\"unster: Dynamics–Geometry–Structure, through SFB 1442, and by the ERC Advanced Grant
834267 - AMAREC.
The first named author was supported by a research grant of 
the Swedish Research Council.
The second named author was supported by the project G085020N 
funded by the Research Foundation Flanders (FWO), and by the generosity of Eric and Wendy Schmidt by recommendation of the Schmidt Futures program.
The fifth named author was supported by European Union's Horizon 2020 research and innovation program under the Marie Sk\l{}odowska-Curie grant agreement No.~891709.
}

\author[Gardella]{Eusebio Gardella}
\address{Eusebio Gardella
Department of Mathematical Sciences, Chalmers University of
Technology and University of Gothenburg, Gothenburg SE-412 96, Sweden.}
\email{gardella@chalmers.se}
\urladdr{www.math.chalmers.se/~gardella}

\author[Geffen]{Shirly Geffen}
\address{Shirly Geffen,
Mathematisches Institut, Fachbereich Mathematik und Informatik der
Universit\"at M\"unster, Einsteinstrasse 62, 48149 M\"unster, Germany.}
\email{sgeffen@uni-muenster.de}
\urladdr{https://shirlygeffen.com/}

\author[Kranz]{Julian Kranz}
\address{Julian Kranz,
Mathematisches Institut, Fachbereich Mathematik und Informatik der
Universit\"at M\"unster, Einsteinstrasse 62, 48149 M\"unster, Germany.}
\email{julian.kranz@uni-muenster.de}
\urladdr{https://www.uni-muenster.de/IVV5WS/WebHop/user/j\_kran05/}

\author[Naryshkin]{Petr Naryshkin}
\address{Petr Naryshkin,
Mathematisches Institut, Fachbereich Mathematik und Informatik der
Universit\"at M\"unster, Einsteinstrasse 62, 48149 M\"unster, Germany.}
\email{pnaryshk@uni-muenster.de}
\urladdr{http://petrnaryshkin.wordpress.com}

\author[Vaccaro]{Andrea Vaccaro}
\address{Andrea Vaccaro, Mathematisches Institut, Fachbereich Mathematik und Informatik der
Universit\"at M\"unster, Einsteinstrasse 62, 48149 M\"unster, Germany.}
\email{avaccaro@uni-muenster.de}
\urladdr{https://sites.google.com/view/avaccaro}

\begin{document}
 
\begin{abstract}

We introduce the notion of tracial amenability for actions of discrete groups on unital, tracial
C$^*$-algebras, as a weakening of amenability where all the 
relevant approximations are done in the uniform trace norm.
We characterize tracial amenability with various equivalent conditions,
including topological amenability of the induced action on the trace space. 
Our main result concerns the structure of crossed products: for groups containing the free group $F_2$, we show that outer,
tracially amenable actions on simple, unital, $\mathcal{Z}$-stable C$^*$-algebras
always have purely infinite crossed products. 
Finally, we give concrete examples of tracially amenable actions of 
free groups on simple, unital AF-algebras. 
\end{abstract}

\maketitle

\section{Introduction}

The theory of amenable actions on C$^*$-algebras is an important tool for studying approximation properties of crossed products. Already initiated in \cite{AD_systemes_1987}, this topic has recently received a lot of attention after it gained new impetus in \cite{BusEchWil22} (see also \cite{BusEchWil20,AbadieBussFerraro,BeardenCrann,Ozawa_Suzuki_2020}). As recently established in \cite{Ozawa_Suzuki_2020},
the notion of amenability is equivalent to the so-called \emph{quasicentral approximation property} (QAP) from \cite{BusEchWil20}. The QAP 
is a versatile tool making powerful averaging techniques accessible, with evidence being the classification of amenable, outer actions on Kirchberg algebras \cite{Gabe_Szabo_2022} or the equivariant $\mathcal O_2$-absorption theorem \cite{Suz_Equivariant_2021}.

The motivation for this paper originates from Elliott's classification program, a long-time endeavour in the theory of C$^*$-algebras
aiming to classify nuclear C$^*$-algebras by $K$-theoretic and tracial data. Elliott's program is now considered to be essentially
completed,
with unital, simple, separable, nuclear, $\mathcal Z$-stable C$^*$-algebras satisfying the Universal Coefficient Theorem (UCT) of Rosenberg
and Schochet \cite{RosSch_kunneth_1987} being classified up to isomorphism by their Elliott invariant (see 
\cite{Win_structure_2018} for an exhaustive bibliography on the matter). 
In the rest of this introduction,
and although this terminology is not standard,
algebras satisfying all these assumptions will be called \emph{classifiable}. 
In view of this recent progress, an important further step is to
identify prominent classes of 
C$^*$-algebras that satisfy the assumptions of the classification
theorem. This problem has attracted
considerable attention in connection to crossed products, 
and most of the work in the literature has focused on actions
of discrete (and usually amenable) groups 
on C$^*$-algebras that are either abelian or simple.
This work deals with the latter setting.

If $\alpha\colon G\to\Aut(A)$ is an action of a discrete group $G$ on a simple, unital C$^*$-algebra $A$, then 
the crossed product $A\rtimes_r G$ is simple whenever $\alpha$ is outer (see \cite{Kishimoto}). If $A$ is moreover nuclear, then $A\rtimes_r G$ is nuclear if and only
if $\alpha$ is amenable (see \cite{AD_systemes_1987}). 
Whether $A\rtimes_rG$ satisfies the UCT is a very subtle question:
the answer is always ``yes'' if $A$ satisfies the UCT, and $G$ is torsion-free with
the Haagerup property (see \cite{MeyerNest, HigsonKasparov}), but the problem is equivalent to the UCT question for torsion groups 
(see \cite[Example~23.15.12~(d)]{Black_K} and \cite{BarSza_cartanII_2020}). Therefore, modulo the UCT, classifiability of $A\rtimes_r G$ for an amenable, outer action on a simple, unital, nuclear C$^*$-algebra reduces to proving $\mathcal Z$-stability
for $A\rtimes_r G$.

The problem of establishing $\mathcal{Z}$-stability for 
$A\rtimes_rG$ when $A$ is simple and nuclear has been largely investigated in the case where $G$ is amenable and $A$ is $\mathcal{Z}$-stable. 
In this case,
$\mathcal{Z}$-stability of $A \rtimes_r G$ is conjectured to always 
hold (see \cite[Conjecture~A]{Sza_Equivariant_2021}), and
this has been verified in full generality if $A$ is purely infinite (\cite{szabo:equivariantKP}), and under various degrees
of generality when $A$ is stably finite and its trace space $T(A)$ is a Bauer simplex with finite dimensional boundary (\cite{matuisato:1, matuisato:2, sato, GHV_2022, Wouters_22}). Further progress has been recently obtained in \cite{ESPA}, where $\mathcal{Z}$-stability
of $A \rtimes_r G$ is obtained in a number of cases where the boundary of $T(A)$ is compact but not necessarily finite 
dimensional. On the other hand, $\mathcal{Z}$-stability may fail 
for actions of nonamenable groups; see \cite[Theorem~B]{GarLup_group_2021}.

In this paper, we study amenable actions of \emph{nonamenable} groups on
simple C$^*$-algebras, inspired by the results in \cite{GGKN22} for commutative $C^*$-algebras. 
Our original motivation was to show that the crossed product
of an outer, amenable action of a nonamenable group 
on a simple, unital, nuclear, 
$\mathcal{Z}$-stable C$^*$-algebra is automatically purely 
infinite and simple. 
For groups containing the free group $F_2$, this follows from 
\autoref{cor:main}.
As it turns out, the assumptions on both the algebra and the 
action can be weakened significantly. For once, we do not need
to assume $A$ to be nuclear or even $\mathcal{Z}$-stable (see
\autoref{cor:main} for the minimal set of assumptions).
More important for this work is the fact that amenability of
$\alpha$ is also stronger than necessary. This observation led
us to identify and isolate the following notion:

\begin{dfnintro}[\autoref{df:TracialAmenAct}]\label{dfintro}
An action $\alpha\colon G\to\Aut(A)$ of a countable, discrete group $G$
on a separable, unital, tracial C$^*$-algebra $A$ is said to be
\emph{tracially amenable} if there is a sequence 
$(\xi_n)_{n\in\N}$ of finitely supported functions 
$\xi_n\colon G\to A$ with $\|\xi_n\|\leq 1$ such that  
\[\lim_{n\to\I} \|\xi_n a-a\xi_n\|_{2,u}=
\lim_{n\to\I}\left\|\langle \xi_n,\xi_n\rangle-1\right\|_{2,u} 
=\lim_{n\to\I}\left\|\widetilde{\alpha}_g(\xi_n)-\xi_n\right\|_{2,u}=0\]
for all $a\in A$ and all $g\in G$.
\end{dfnintro}

In the above definition, we denote by $\widetilde{\alpha}$ the 
diagonal action on $\ell^2(G,A)$ with left translation on $G$ and 
$\alpha$ on $A$, 
and we denote by $\|\cdot\|_{2,u}$ the so-called
\emph{uniform trace norm} on $\ell^2(G,A)$; see the comments before 
\autoref{lma:Cauchy-Schwarz}. 
Thanks to the characterizations obtained in \cite{Ozawa_Suzuki_2020}, the usual notion of amenability is obtained by replacing the 
uniform trace norm by the usual Hilbert C$^*$-module-norm on $\ell^2(G,A)$.
Since $\|\cdot\|_{2,u}\leq \|\cdot \|$, it follows that any amenable action is tracially amenable.

\autoref{S.3} is devoted to obtaining several characterizations
of tracial amenability, inspired by the work done in \cite[Theorem 3.2, Theorem 4,4]{Ozawa_Suzuki_2020}. We reproduce some of them below:

\begin{thmintro}[\autoref{thm:TracAm}]
Let $\alpha\colon G\to \mathrm{Aut}(A)$ be an action of a countable, discrete, exact group $G$ on a separable, unital, tracial C$^*$-algebra $A$.
The following are equivalent: 
\begin{enumerate}
		\item The action $\alpha$ is tracially amenable. \label{item:trac amen}
		\item \label{main:i2} There is a sequence $(\xi_n)_{n\in\N}$ of finitely supported functions 
$\xi_n\colon G\to A$ with $\|\xi_n\|\leq 1$ such that  
\[\lim_{n\to\I}\left\|\langle \xi_n,\xi_n\rangle-1\right\|_{2,u} 
=\lim_{n\to\I}\left\|\widetilde{\alpha}_g(\xi_n)-\xi_n\right\|_{2,u}=0\]
for all $g\in G$.
 \label{item:noncentral}
		\item The induced action $\alpha^{\omega}\colon G\to\mathrm{Aut}(A^\omega\cap \iota(A)')$ on the tracial ultrapower is amenable. \label{item:amenable ultrapower}	
		\item \label{main:i4} The induced action $G\curvearrowright T(A)$ is topologically amenable. \label{item:T(A)}
		\item \label{main:i5} The induced action $G\curvearrowright \overline{\partial_e T(A)}$ is topologically amenable. \label{item:dT(A)}		
		\item The induced action $\alpha_{\mathrm{fin}}^{\ast\ast} \colon G\to \Aut(A\dfin)$ is von Neumann-amenable. \label{item:vN}
	\end{enumerate}
\end{thmintro}

Some of the above conditions look similar to 
analogous characterizations of amenability on C$^*$-algebras, while some do not admit
a counterpart in that setting. For example,
the difference between \autoref{dfintro} and item \eqref{main:i2} above is that
we do not require approximate centrality in \eqref{main:i2}. Also, \eqref{main:i4} and \eqref{main:i5} do not have analogues in
the setting of amenable actions, since amenability of $G\curvearrowright S(A)$ does not imply amenability of $\alpha$.

We study the structure of crossed products in \autoref{S.4}.
For groups containing the free group $F_2$, we show that
tracially amenable actions give rise to purely infinite crossed products:

\begin{thmintro}[\autoref{cor:main}]\label{thmintro A}
Let $G$ be a countable, discrete group containing the free group $F_2$, let 
$A$ be a simple, separable, unital, stably finite,
nuclear, $\mathcal{Z}$-stable C$^*$-algebra, and 
let $\alpha\colon G\to \Aut(A)$ be a tracially amenable, 
outer action. Then $A\rtimes_r G$ is a unital, simple, purely infinite C$^*$-algebra. 
\end{thmintro}

As mentioned before, the requirements on $A$ can be weakened, and 
we in particular do not need to assume $A$ to be either nuclear
or $\mathcal{Z}$-stable (see the statement of \autoref{cor:main} for the precise assumptions on $A$). The condition on the group can also be 
relaxed, and it suffices to assume that $G$ has what we call
\emph{weak paradoxical towers}; see \autoref{df:WeakParTwrs}. 
To obtain \autoref{thmintro A}, we show in \autoref{prop:CompPurInf}
that it suffices to
prove that any action as in the statement satisfies what 
Bosa, Perera, Wu and Zacharias call \emph{dynamical 
strict comparison} (\autoref{df:DynCom}). That this is the case 
is shown in \autoref{thm:Mainthm}, by exploiting the tension 
between tracial amenability of $\alpha$ and the existence of 
weak paradoxical towers in $G$. 

\autoref{thmintro A} is new even if $\alpha$ is amenable. In this
setting, $A\rtimes_r G$ is nuclear and therefore a Kirchberg algebra,
so in particular $\mathcal{O}_\infty$-stable (and, therefore, also $\mathcal{Z}$-stable).
If, moreover, $G$ is torsion-free and has the Haagerup property 
(for example, $G=F_n$; see \cite[Definition 12.2.1]{BrownOzawa}),
and $A$ satisfies the UCT,
then $A\rtimes_r G$ also satisfies the UCT and thus is completely determined by its $K$-theory by \cite{Phi_classification_2000}; see \autoref{cor:GrpsContF2}. 
Thus, the assumptions of outerness and amenability on $\alpha$ do in fact 
guarantee classifiability of the crossed product. There is, however,
a drawback:

\begin{pbmintro} \label{prob:nasf}
Are there any amenable actions of nonamenable groups on simple, unital, stably finite C$^*$-algebras?
\end{pbmintro}

The above problem has recently 
attracted a fair amount of attention, and 
to the best of our knowledge it remains open. 
If one drops unitality of the algebra, an example
has been constructed in \cite{Suzuki_Every_2022}.

The fact that \autoref{prob:nasf} remains open highlights another
advantage of focusing on \emph{tracially} amenable actions throughout: 
namely, we can construct several examples of actions satisfying
the assumptions of \autoref{thmintro A}:

\begin{egintro}[\autoref{eg:TracAmenEll}; \autoref{prop:ExOnAFAlg}]\label{egintro}
For every $n\geq 2$, 
there exist outer, tracially amenable actions of the free group 
$F_n$ on stably finite,
classifiable C$^*$-algebras, including actions on simple, unital
AF-algebras.
\end{egintro}

We do not know if the 
actions that we construct here are amenable.

\subsection*{Acknowledgements}
We are indebted to Sam Evington for his help with the proof of \autoref{lma:Sam}, to Thierry Giordano, to Mikael R{\o}rdam for suggesting \autoref{prop:ExOnAFAlg}, and to Gábor Szabó and Stuart White for helpful comments throughout the development of this project.
We are very grateful to the referee for their thorough reading of the manuscript as well as for the numerous helpful suggestions and corrections.
The first and the second author would like to thank the University of M\"unster for the hospitality during the visit in which part of this research was done.
The third author would like to thank the Chalmers University of Technology and the University of Gothenburg for their hospitality during his stay there.

\section{Tracially amenable actions} \label{S.3}

Amenable actions of discrete groups on C$^*$-algebras were introduced in \cite{AD_systemes_1987} and later extended to locally compact groups in \cite{BusEchWil22}. In this section, we define a tracial analogue of amenability for actions of discrete groups (\autoref{df:TracialAmenAct}). The motivation to study this notion is two-fold. 
On the one hand, tracial amenability allows us 
to obtain strong structural results for the crossed product;
see Subsection~\ref{S.4}. On the other hand, and unlike for amenable actions, 
it is possible to construct many tracially amenable actions of nonamenable groups on unital, simple, stably finite C$^*$-algebras, by means of classification results (see Subsection~\ref{ss.classi}).

Fix a unital C$^*$-algebra $A$ and a discrete group $G$.
For an action $\alpha\colon G\to \mathrm{Aut}(A)$, we denote by $\ell^2(G,A)$ the completion of $C_c(G,A)$ with respect to the norm
$\|\cdot\|$ induced by the $A$-valued inner product given by
\[\langle \xi,\eta\rangle =\sum_{g\in G}\xi(g)^*\eta(g)\] 
for all $\xi,\eta\in C_c(G,A)$.
We equip $\ell^2(G,A)$ with the $A$-bimodule structure given by pointwise multiplication and the diagonal $G$-action $\widetilde{\alpha}$ given by $\widetilde{\alpha}_g(\xi)(h)= \alpha_g(\xi(g^{-1}h))$, for 
all $\xi\in \ell^2(G,A)$, and all $g,h\in G$.

Recall that a tracial state (often called simply ``trace'' in this work)
on $A$ is a continuous linear functional $\tau\colon A\to \mathbb{C}$
satisfying $\tau(1)=1$, $\tau(A_+)\subseteq [0,\infty)$ and $\tau(ab)=\tau(ba)$
for all $a,b\in A$.
We write $T(A)$ for the compact convex space of all traces on $A$.
Given a trace $\tau\in T(A)$ and $a \in A$, we denote by $\|a\|_{2,\tau}= \tau(a^*a)^{\frac 1 2}$ the associated \emph{$2$-seminorm} on $A$. We also denote by $\|\cdot \|_{2,\tau}$ the induced seminorm on $\ell^2(G,A)$, given by 
$\|\xi\|_{2,\tau}= \tau(\langle \xi,\xi\rangle)^{\frac{1}{2}}$,
for all $\xi\in \ell^2(G,A)$.
In order to see that this is indeed a seminorm, note that $\|\xi\|_{2,\tau}$ agrees with the norm of the element $\xi\otimes \xi_\tau\in \ell^2(G,A)\otimes_{\pi_\tau}H_\tau$ where $(\pi_\tau,H_\tau,\xi_\tau)$ is the GNS construction of $\tau$. 
The \emph{uniform $2$-seminorm} on $A$ (respectively, on $\ell^2(G,A)$) is given by $\|\cdot\|_{2,u}= \sup_{\tau\in T(A)}\|\cdot\|_{2,\tau}$. 
The following observation, which is a variant of the Cauchy-Schwarz 
inequality, will be used repeatedly.

\begin{lma}\label{lma:Cauchy-Schwarz}
Let $A$ be a unital C$^*$-algebra and let $\tau\in T(A)$.
Then:
\be\item For every $a\in A$, we have
$|\tau(a)|\leq \|a\|_{2,\tau}\leq \|a\|$. 
\item For every $\xi,\eta\in \ell^2(G,A)$ we have  $\|\langle \xi,\eta\rangle\|_{2,\tau}\leq \|\xi\|\|\eta\|_{2,\tau}\leq \|\xi\|\|\eta\|$.
\ee 
\end{lma}
Part~(1) is immediate. For part (2), 
by \cite[Proposition~1.1]{Lance} we have 
\[\langle \xi,\eta\rangle^*\langle \xi,\eta\rangle \leq \|\langle \xi,\xi\rangle \|\langle \eta,\eta\rangle.\]
Applying $\tau$ and taking the square roots gives the first inequality,
while the second one is immediate.

The following definition is inspired by the \emph{quasicentral approximation property} (QAP) from \cite[Definition 3.1]{BusEchWil20},
which was shown to be equivalent to amenability for actions on C$^*$-algebras in \cite[Theorem~3.2]{Ozawa_Suzuki_2020}. In view of this
equivalence, we directly present the tracial analogue of \cite[Definition 3.1]{BusEchWil20} under the name of tracial amenability.

\begin{df}\label{df:TracialAmenAct}
An action $\alpha\colon G\to \Aut(A)$ of a discrete group $G$ on a unital, C$^*$-algebra $A$ with $T(A)\neq \emptyset$ 
is called \emph{tracially amenable} if
for all finite subsets $F \subseteq A$ and $K \subseteq G$, and 
for every $\varepsilon > 0$, there exists
$\xi \in C_c(G,A)$ satisfying 
\begin{enumerate}
		\item $\left\|\xi \right\|\leq 1$ (in the Hilbert C$^*$-module-norm of $\ell^2(G,A)$); \label{item:dfTA contr}
		\item $\left\|\xi a-a \xi\right\|_{2,u} < \varepsilon$, for all $a\in F$; \label{item:dfTA comm}
		\item $\left\|\langle \xi,\xi\rangle-1\right\|_{2,u} < \varepsilon$; \label{item:dfTA 1}
		\item $\left\|\widetilde{\alpha}_g(\xi)-\xi\right\|_{2,u} < \varepsilon$, for all $g\in K$. \label{item:dfTA inv}
	\end{enumerate}
\end{df}

\begin{rem} \label{rem:alphathings}
	Note that conditions \eqref{item:dfTA 1} and \eqref{item:dfTA inv} in \autoref{df:TracialAmenAct} can be equivalently replaced by the condition $\|\langle \xi,\widetilde{\alpha}_g(\xi)\rangle -1\|_{2,u}<\varepsilon$ for all $g\in K$.
	If $G$ is countable and discrete, and $A$ is separable, \autoref{df:TracialAmenAct} is furthermore equivalent to the existence of a sequence $(\xi_n)_{n\in \N}$ in $C_c(G,A)$ such that the terms in items \eqref{item:dfTA comm}-\eqref{item:dfTA inv} converge to zero along $n\to \infty$, for every $a\in A$ and $g\in G$.
\end{rem}

The definition of the QAP in \cite[Definition 3.1]{BusEchWil20} differs from \autoref{df:TracialAmenAct}
in that all estimates are formulated in the C$^*$-norm, rather than in the tracial seminorm. Since the C$^*$-norm dominates $\|\cdot\|_{2,u}$, 
it follows from \cite[Theorem~3.2]{Ozawa_Suzuki_2020} that every amenable action is tracially amenable. 

Our next goal is to obtain 
several characterizations of tracial amenability; see 
\autoref{thm:TracAm}. For this, we need some preparation.
Given an action $\alpha\colon G\to \mathrm{Aut}(A)$, a function $\theta\colon G\to A$ is said to be of \emph{positive type} with respect to $\alpha$ if for every finite set $K\subseteq G$, the matrix $(\alpha_g(\theta(g^{-1}h)))_{g,h\in K}\in M_{|K|}(A)$ is positive. 
We reproduce here two results from \cite{AD_systemes_1987} in a way
that will be more convenient for us later.
The center of a C$^*$-algebra $A$ is denoted $Z(A)$.

\begin{thm}\cite[Theorem~3.3, Theorem~4.9]{AD_systemes_1987}\label{thm:positive type}
	Let $\alpha\colon G\to \Aut(A)$ be an action of a discrete group $G$ 
	on a unital C$^*$-algebra $A$. Suppose that for every finite subset $K\subseteq G$ and every $\varepsilon>0$, there exists a finitely supported positive type function $\theta\colon G\to Z(A)$ (with respect to $\alpha$) satisfying $\theta(e)\leq 1$ and $\|\theta(g)-1\|<\varepsilon$ for all $g\in K$. 	
	Then $\alpha$ is amenable. If $A$ is commutative, the converse holds as well. 
\end{thm}

Given a unital C$^*$-algebra $A$, its tracial state space $T(A)$ is a convex compact topological space, and we
denote by $\partial_e T(A)$ the set of its extremal points. 

For a fixed free ultrafilter $\omega\in \beta \N\setminus \N$, we denote by $A^\omega$ the quotient of $\ell^\infty(\N,A)$ by the ideal of sequences $(a_n)_{n\in \N}$ satisfying $\lim_{n\to \omega}\|a_n\|_{2,u}=0$. This is the \emph{tracial ultrapower} of $A$. There exists a natural unital
homomorphism $\iota \colon A \to A^\omega$, which maps each element in $A$ to the class of the corresponding constant sequence. Note
that the map $\iota$ is not necessarily injective, unless traces separate positive elements (which is always the case if $A$ is simple and tracial).
We denote by $A^\omega \cap \iota(A)'$ the 
C$^*$-algebra consisting of all elements in $A^\omega$ that commute with 
(the images of) all constant sequences.
Any action $\alpha \colon G \to \text{Aut}(A)$ canonically induces 
an action on $A^\omega$ and also on $A^\omega \cap \iota(A)'$,
by acting coordinatewise. With a slight abuse of notation, we shall denote both actions by $\alpha^\omega$.

The \emph{finite part} $A\dfin$ of the bidual $A\dd$ of a C$^*$-algebra $A$ can be identified with the weak closure of $A$ with respect to the sum of all GNS-representations for all traces. Namely,
	\[A\dfin= \left(\bigoplus\nolimits_{\tau\in T(A)}\pi_\tau\right)(A)''\subseteq \mathcal B\left(\bigoplus\nolimits_{\tau\in T(A)}H_\tau\right).\]
Every action $\alpha$ on $A$ canonically extends
to an action $\alpha_{\mathrm{fin}}^{\ast\ast}$ on $A\dfin$.

An action $\alpha\colon G\to \Aut(M)$ of a discrete group $G$ on a von  Neumann algebra $M$ is called \emph{von Neumann-amenable} if 
there exists a net $(\xi_\lambda)_{\lambda\in\Lambda}$ in $C_c(G,Z(M))$ such that $\langle \xi_\lambda, \xi_\lambda\rangle=1$, for all $\lambda\in \Lambda$, and the net $(\langle \xi_\lambda,\widetilde{\alpha}_g(\xi_\lambda)\rangle)_{\lambda\in\Lambda}$ converges to $1\in Z(M)$ pointwise on $G$ in the ultraweak topology (see \cite[Theorem~1.1]{BeardenCrann}).

In the following result we give several conditions that are equivalent
to tracial amenability of an action.
Some of the conditions below look similar to 
analogous characterizations of amenability obtained in (and, indeed, motivated by) \cite[Theorem 3.2, Theorem 4,4]{Ozawa_Suzuki_2020}. Others do not have
a counterpart in that setting and reveal phenomena that are very special to tracial amenability.
 We discuss this in more detail
after the proof of \autoref{thm:TracAm}; see also \autoref{rem:AmenSA}.

\begin{thm}\label{thm:TracAm}
Let $\alpha\colon G\to \mathrm{Aut}(A)$ be an action of a discrete group $G$ on a unital, separable C$^*$-algebra $A$ with $T(A)\neq\emptyset$. Consider the following conditions:
\begin{enumerate}
\item The action $\alpha$ is tracially amenable. \label{item:trac amen}
\item For $\varepsilon>0$ and any finite set $K\subseteq G$, there exists $\xi\in C_c(G,A)$ such that 
\begin{enumerate}
\item $\|\xi\|\leq 1$;
\item $\|\langle \xi,\xi\rangle -1\|_{2,u}<\varepsilon$;
\item $\|\widetilde{\alpha}_g(\xi)-\xi\|_{2,u}<\varepsilon$ for all $g\in K$.
\end{enumerate} \label{item:noncentral}
\item For any separable, unital C$^*$-algebra $C$ with an action $\gamma \colon G \to \mathrm{Aut}(C)$, there is an equivariant, unital, completely positive map $\Phi\colon (C,\gamma)\to (A^\omega\cap \iota(A)',\alpha^\omega)$. \label{item:map from C} 
\item There is a sequence of 
functions $\theta_n\colon G\to A^\omega\cap \iota(A)'$ of finite support and of positive type with respect to $\alpha^\omega$, which satisfy $\theta_n(e)\leq 1$ for all $n\in\mathbb{N}$, and $\lim_{n \to \infty}\|\theta_n(g)-1\| = 0$, for all $g\in G$. 	\label{item:positive type}
\item The induced action $\alpha^{\omega}\colon G\to\mathrm{Aut}(A^\omega\cap \iota(A)')$ is amenable. \label{item:amenable ultrapower}	
\item The induced action $G\curvearrowright T(A)$ is topologically amenable. \label{item:T(A)}
\item The induced action $G\curvearrowright \overline{\partial_e T(A)}$ is topologically amenable. \label{item:dT(A)}		
\item The induced action $\alpha_{\mathrm{fin}}^{\ast\ast} \colon G\to \Aut(A\dfin)$ is von Neumann-amenable. \label{item:vN}
\end{enumerate}
Then \eqref{item:trac amen} $\Leftrightarrow$ \eqref{item:noncentral} $\Leftrightarrow$ \eqref{item:T(A)} $\Leftrightarrow$ \eqref{item:dT(A)} $\Leftrightarrow$ \eqref{item:vN}. If moreover $G$ is countable and exact then all of the above conditions are equivalent. 
\end{thm}

For the proof we need the following well-known lemmas.

\begin{lma} \label{lma:lifting contractions}
	Let $A$ and $B$ be unital C$^*$-algebras, let $\pi\colon A\to B$ be a surjective  homomorphism, and let $F$ be a finite set. Denote the induced map $\ell^2(F,A)\to \ell^2(F,B)$ again by $\pi$. Given $\eta\in \ell^2(F,B)$, there exists $\xi\in \ell^2(F,A)$ such that $\pi(\xi)=\eta$ and $\|\xi\|=\|\eta\|$.
\end{lma}

\begin{proof}
	Note that there is an isometric embedding 
	$\ell^2(F,A)\hookrightarrow M_{|F|}(A)$ given by
	\[(a_1,\dotsc,a_n)\mapsto \begin{pmatrix}
	a_1 	&0 		&\dotsb 	&0		\\ 
	\vdots	&\vdots &\ddots		&\vdots \\
	a_n		&0 		&\dotsb		&0		\\
	\end{pmatrix},\]
	and an analogous embedding $\ell^2(F,B)\hookrightarrow M_{|F|}(B)$.
	By a standard functional calculus argument (see, for example, \cite[Lemma~17.3.3]{WeggeOlsen}), there is an element $\xi'\in M_{|F|}(A)$ satisfying $\pi\otimes\text{id}_{M_{| F |}}(\xi')=\eta$ and $\|\xi'\|=\|\eta\|$. We write $p=\diag(1,0,\dotsc,0)\in M_{|F|}(A)$. Then $\xi\coloneqq \xi' p\in \ell^2(F,A)$ has the desired properties. 
\end{proof}

\begin{lma}\label{lma:cts-affine}
Let $K$ be a compact convex set and denote by $\mathrm{Aff}_c(K)\subseteq C(K)$ the set of all continuous affine functions $K\to \mathbb R$. Then the restriction of the weak topology of $C(K)$ to $\mathrm{Aff}_c(K)$ coincides with the topology of pointwise convergence.
\end{lma}
\begin{proof}
	This follows from the fact that for every $f\in \mathrm{Aff}_c(K)$ and any Radon probability measure $\mu$ on $K$, we have 
		\[\int_Kfd\mu=f(\beta(\mu)),\]
	where $\beta(\mu)\in K$ denotes the barycenter of $\mu$ (see \cite[Lemma IV.6.3]{takesaki}). 
\end{proof}

\begin{proof}[Proof of \autoref{thm:TracAm}]
We prove the following set of implications where, in the implications labeled by dashed arrows, we additionally assume that $G$ is countable and exact:
\[
\begin{tikzcd}[column sep=2 em, row sep=0.2 em]
	\eqref{item:amenable ultrapower}\arrow[r,Leftrightarrow,dashed]  &\eqref{item:positive type}\arrow[r,Rightarrow,dashed] &\eqref{item:trac amen}\arrow[dd,Rightarrow] \arrow[ddl,Rightarrow,dashed] &\eqref{item:vN}\arrow[l,Rightarrow] & \\
	&&&&\eqref{item:dT(A)}\arrow[ul,Rightarrow]\\
	&\eqref{item:map from C}\arrow[uu,Rightarrow,dashed]&\eqref{item:noncentral}\arrow[r,Rightarrow] &\eqref{item:T(A)} \arrow[ur,Rightarrow]&
\end{tikzcd}
\]

$\eqref{item:trac amen}\Rightarrow \eqref{item:map from C}$. Fix a state $\varphi\in S(C)$. Since $G$ is countable and $A$ is separable, we can find a sequence $(\xi_n)_{n\in \N}$ in $C_c(G,A)$ satisfying 
\begin{enumerate}[label=(a.\arabic*)]
	\item $\|\xi_n\|\leq 1$ for all $n\in \N$;
	\item $\lim_{n\to \infty}\|\xi_na-a\xi_n\|_{2,u}=0$ for all $a\in A$;
	\item $\lim_{n\to \infty}\|\langle \xi_n,\xi_n\rangle -1\|_{2,u}=0$;
	\item $\lim_{n\to \infty}\|\widetilde{\alpha}_g(\xi_n)-\xi_n\|_{2,u}=0$ for all $g\in G$. 
\end{enumerate} 
We define a map $\widetilde{\Phi}=(\Phi_n)_{n\in\N}\colon C\to \ell^\infty(\N,A)$ by setting
\[\Phi_n(c)\coloneqq \sum_{g\in G}\varphi(\gamma_{g^{-1}}(c))\xi_n(g)^*\xi_n(g),\]
for $n\in \N$ and $c\in C$. (Note that $\sup_{n\in\mathbb{N}}\|\Phi_n\|\leq 1$ by (a.1) above, so that the resulting map $\widetilde{\Phi}$ does indeed take values in $\ell^{\infty}(\mathbb{N},A)$.) Denote by $\Phi\colon C\to A^\omega$ the 
composition of $\widetilde{\Phi}$ with the canonical quotient map 
$\ell^\infty(\N,A)\to A^\omega$.
We claim that $\Phi$ is an equivariant, unital, completely positive map,
and that its image is contained in $A^\omega\cap \iota(A)'$. 
The fact that $\Phi$ is unital follows from the equality ${\|\Phi_n(1)-1\|_{2,u}=\|\langle \xi_n,\xi_n\rangle -1\|_{2,u}}$ together with (a.3). As each $\Phi_n$ is completely positive, the same is true for $\widetilde{\Phi}$ and thus also for $\Phi$. To see that $\Phi$ takes values in $A^\omega \cap \iota(A)'$, fix $a\in A$ and $c\in C$. Given
$d\in C$, denote by $T_d\in \mathcal{L}(\ell^2(G,A))$ the operator 
given by pointwise multiplication with the function $g\mapsto \varphi(\gamma_{g^{-1}}(d))$, and note that $\|T_d\|\leq \|d\|$. 
Using the Cauchy-Schwarz inequality (see part~(2) of \autoref{lma:Cauchy-Schwarz}) and $\|\xi_n\|\leq 1$ at the last step, we get 
\begin{align*}
\big\|a\Phi_n(c)-\Phi_n(c)a\big\|_{2,u} &=\big\|a\langle \xi_n,T_c\xi_n\rangle -\langle \xi_n,T_c\xi_n\rangle a \big\|_{2,u} \\
&= \big\|\langle \xi_n a^*,T_{c}\xi_n\rangle -\langle T_{c^*}\xi_n,\xi_n a\rangle \big\|_{2,u}\\
&\leq \big\|\langle \xi_n a^*-a^*\xi_n,T_{c} \xi_n\rangle \big\|_{2,u}+\big\|\langle T_{c^*} \xi_n,a\xi_n-\xi_na\rangle \big\|_{2,u}\\
&\leq \|c\|\big(\|\xi_n a^*-a^*\xi_n\|_{2,u}+\|a\xi_n -\xi_na\|_{2,u}\big),
\end{align*}
with the latter term converging to zero as $n \to \infty$, by (a.2).
To see that $\Phi$ is equivariant, fix $g\in G$ and $c\in C$. Using again the Cauchy-Schwarz inequality and $\|\xi_n\|=\|\widetilde{\alpha}_g(\xi_n)\|\leq 1$ at the last step, we obtain
\begin{align*}
\big\|\alpha_g(\Phi_n(c))&-\Phi_n(\gamma_g(c))\big\|_{2,u} \\
&=\Big\|\sum_{h\in G}\varphi(\gamma_{h^{-1}g}(c))\Big(\alpha_g(\xi_n(g^{-1}h)^*\xi_n(g^{-1}h)) -\xi_n(h)^*\xi_n(h)\Big)\Big\|_{2,u}\\
&=\Big \|\langle  \widetilde{\alpha}_g(\xi_n),T_{\gamma_g(c)}\widetilde{\alpha}_g(\xi_n)\rangle -\langle \xi_n,T_{\gamma_g(c)}\xi_n\rangle \Big\|_{2,u}\\
&\leq 2\|c\|\|\xi_n-\widetilde{\alpha}_g(\xi_n)\|_{2,u},
\end{align*}
with the latter term converging to zero as $n \to \infty$, by (a.4). This proves \eqref{item:map from C}.

$\eqref{item:map from C}\Rightarrow \eqref{item:positive type}$. Recall that since $G$ is countable and exact, there is an amenable action of $G$ on a compact, metrizable space $X$ (see \cite[Theorem~5.1.7]{BrownOzawa}). By \autoref{thm:positive type}, there exists a sequence of finitely supported positive type functions $\theta_n\colon G\to C(X)$ with respect to $\alpha$ satisfying $\theta_n(e)\leq 1$ and 
$\lim_{n\to \infty}\|\theta_n(g)-1\|=0$ for all $g \in G$.
If $\Psi\colon C(X)\to A^\omega\cap \iota(A)'$ is an equivariant, unital, completely positive map, then $\Psi\circ \theta_n\colon G\to A^\omega\cap \iota(A)'$,
for $n \in \mathbb{N}$, gives a sequence of positive type functions with respect to $\alpha^\omega$ with the desired properties.

$\eqref{item:positive type}\Rightarrow\eqref{item:amenable ultrapower}$. We prove amenability of the action $\alpha^\omega \colon
G \to \text{Aut}(A^\omega \cap \iota(A)')$ by showing that it satisfies the QAP, which can be stated as
our \autoref{df:TracialAmenAct} once every occurrence of the tracial norm is replaced with the C$^*$-norm on $A^\omega$
(see \cite[Definition 3.1]{BusEchWil20}), and then resorting
to \cite[Theorem~3.2]{Ozawa_Suzuki_2020}. Fix a finite subset $K \subseteq G$ and $\varepsilon>0$. By assumption, there is a finitely supported positive type function $\theta\colon G\to A^\omega\cap \iota(A)'$ satisfying $\theta(e)\leq 1$ and $\|\theta(g)-1\|<
\varepsilon $ for all $g\in K$. By \cite[Proposition~2.5]{AD_systemes_1987}, there exists $\eta\in \ell^2(G,A^\omega\cap\iota(A)')$ with $\theta(g)=\left\langle \eta, \widetilde{\alpha^\omega_g}(\eta)\right\rangle$ for all $g\in G$. 
By slightly perturbing $\eta$ (and therefore $\theta$), we can assume that  $\eta$ has finite support.
This almost gives the 
QAP for $\alpha^\omega$, except for the condition of almost centrality, which follows from a standard reindexing argument.
To show this, fix a finite set $F \subseteq A^\omega
\cap \iota(A)'$. Applying \autoref{lma:lifting contractions} to the quotient map $\ell^\infty(\N,A)\to A^\omega$, we can find a sequence $(\eta_n)_{n\in \N}$ in $C_c(G,A)\subseteq \ell^2(G,A)$ that represents $\eta$ and satisfies $\|\eta_n\|\leq 1$ and $\supp (\eta_n)\subseteq \supp (\eta)$ for every $n\in \N$. 
Since we have $\|\theta(g)-1\|< \varepsilon $ for all $g\in K$, we can find $(b_n^{(g)})_{n\in \N}\in \ell^\infty(\N,A)$, for $g\in K$, satisfying $\lim_{n\to \omega}\|b_n^{(g)}\|_{2,u}=0$ for all $g\in K$, and 
\begin{equation}\label{eq:eta_n}
	\|\langle \eta_n,\widetilde{\alpha}_g(\eta_n)\rangle -1+b_n^{(g)}\|\leq\varepsilon 
\end{equation} 
for all $n\in\N$ and $g\in K$. Now choose a dense sequence $(a_n)_{n\in \N}$ in $A$, and for each $x\in F$ find a representative $(x_n)_{n\in \N}\in \ell^\infty(\N,A)$. Using that the $(\eta_n)_{n\in \N}$ have uniformly finite support (contained in $\supp(\eta)$) and that each $(\eta_n(g))_{n\in \N}$, for $g\in \supp (\eta)$, represents the element $\eta(g)\in A^\omega\cap \iota(A)'$, we can find inductively an increasing sequence $(k(n))_{n\in \N}$ of natural numbers satisfying 
\begin{enumerate}[label=(b.\arabic*)]
	\item $\|x_n\eta_{k(n)}-\eta_{k(n)}x_n\|_{2,u}<\frac 1 n$, for all $n\in \N$;
	\item $\|a_j\eta_{k(n)}-\eta_{k(n)}a_j\|_{2,u}<\frac 1 n$, for all $n\in \N$ and $j=1,\dotsc,n$;
	\item $\| b_{k(n)}^{(g)} \|_{2,u} < \frac 1 n$, for all $n\in \N$ and $g\in K$.
\end{enumerate}
Then the sequence $\xi\coloneqq ( \eta_{k(n)})_{n\in \N}$ gives rise to a contraction in $\ell^2(G,A^\omega\cap \iota(A)')$ of finite support satisfying $x\xi=\xi x$ for all $x\in F$. 
Moreover,
	\[\big\|\big\langle \xi,\widetilde{\alpha^\omega_g}(\xi)\big\rangle -1\big\|\leq\sup_{n\in \N}\big\|\langle \eta_{k(n)},\widetilde{\alpha}_g(\eta_{k(n)})\rangle-1+{b_{k(n)}^{(g)}}\big\|\stackrel{\eqref{eq:eta_n}}{\leq}\varepsilon,\]
for all $g\in K$. This proves \eqref{item:amenable ultrapower}.

$\eqref{item:amenable ultrapower}\Rightarrow \eqref{item:positive type}$. This follows from associating to an element 
$\xi\in \ell^2(G,A^\omega\cap \iota(A)')$ of finite support, the finitely supported positive type function $\theta_\xi\colon G\to A^\omega\cap \iota(A)'$ given by
	\[\theta_\xi(g)= \left\langle \xi,\widetilde{\alpha^\omega_g}(\xi)\right\rangle\]
for all $g\in G$, and using that amenability of $\alpha^\omega$ is equivalent to the QAP.

$\eqref{item:positive type} \Rightarrow \eqref{item:trac amen}$. 
Let $F\subseteq A$ and $K\subseteq G$ be finite sets and $\varepsilon>0$. 
As in the proof of $\eqref{item:positive type}\Rightarrow\eqref{item:amenable ultrapower}$, we can find an element $\eta\in C_c(G,A^\omega\cap \iota(A)')$ satisfying $\|\eta\|\leq 1$ and 
\begin{equation}\label{eq:eta=1}
\left\|\left\langle \eta,\widetilde{\alpha^\omega_g}(\eta)\right\rangle -1\right\|< \varepsilon,
\end{equation} 
for all $g\in K$, and a sequence $(\eta_n)_{n\in \N}$ in $C_c(G,A)$ representing $\eta$ that satisfies $\|\eta_n\|\leq 1$ and $\supp (\eta_n)\subseteq \supp (\eta)$ for every $n\in \N$. We can thus find a large enough $n_0\in \N$ so that 
\begin{enumerate}[label=(c.\arabic*)]
	\item $\|a\eta_{n_0}-\eta_{n_0}a\|_{2,u}<\varepsilon$, for all $a\in F$;
	\item $\|\langle\eta_{n_0},\widetilde{\alpha}_g(\eta_{n_0})\rangle -1\|_{2,u}\leq \varepsilon$, for all $g\in K$.
\end{enumerate}
This proves \eqref{item:trac amen} (see \autoref{rem:alphathings}).

$\eqref{item:trac amen}\Rightarrow \eqref{item:noncentral}$. This is immediate. 

$\eqref{item:noncentral}\Rightarrow \eqref{item:T(A)}$. Fix a finite set $K\subseteq G$ and $\varepsilon>0$. Choose an element $\xi\in C_c(G,A)$ satisfying $\|\xi\|\leq 1$ and 
	\[\|\langle \xi,\widetilde{\alpha}_g(\xi)\rangle-1\|_{2,u}< \varepsilon,\]
	for $g\in K$. 
	 We define a positive type function $\theta\colon G\to C(T(A))$ by
\[\theta(g)(\tau)\coloneqq \tau(\langle\xi,\widetilde{\alpha}_g(\xi)\rangle)\]
for all $g\in G$ and all $\tau\in T(A)$.
Using part~(1) of \autoref{lma:Cauchy-Schwarz} at the second step, we get 
\begin{align*}
\|\theta(g)-1\|&= \sup_{\tau\in T(A)}\left|\tau(\langle \xi,\widetilde{\alpha}_g(\xi)\rangle -1)\right| \\
&\leq \|\langle \xi,\widetilde{\alpha}_g(\xi)\rangle-1\|_{2,u} < \varepsilon.
\end{align*}
for all $g\in K$. Thus, $G\acton T(A)$ is topologically amenable by \autoref{thm:positive type}.

$\eqref{item:T(A)} \Rightarrow \eqref{item:dT(A)}$. This follows 
by composing the functions $G\to C(T(A))$ coming from amenability of $G\curvearrowright T(A)$ with the equivariant 
quotient map $C(T(A))\to C(\overline{\partial_e T(A)})$.

$\eqref{item:dT(A)}\Rightarrow\eqref{item:vN}$. 
Denote by $B(\partial_e T(A))$ the $C^*$-algebra of Borel bounded functions on $\partial_e T(A)$.
By \cite[Theorem 3]{Ozawa_Dixmier_2013}, there is a unital  homomorphism $\theta \colon B(\partial_e T(A)) \to Z(A\dfin)$.
We claim that $\theta$ is equivariant.
Given $\tau\in T(A)$, identify $\tau$ with its unique normal 
extension to $A^{\ast\ast}_{\mathrm{fin}}$ and write $\mu_\tau$ for 
the unique probability measure on $\partial_eT(A)$ whose barycenter
is $\tau$. 
The map $\theta\colon B(\partial_e T(A)) \to Z(A\dfin)$ obtained in \cite[Theorem 3]{Ozawa_Dixmier_2013}
satisfies
\begin{equation} \label{eq:normal}
\tau(\theta(f)) = \int_{\partial_e T(A)} f(\lambda) \, d\mu_\tau(\lambda),
\end{equation}
for every $f \in B(\partial_e T(A))$ and every 
$\tau\in T(A)$. Condition \eqref{eq:normal} uniquely determines $\theta$, since traces on $A$ (or rather their normal extensions
to $A\dfin$) are in bijective correspondence with the normal states on $Z(A\dfin)$, via the restriction map (see \cite[Theorem~III.2.5.7]{Black_06} and \cite[Theorem~III.2.5.14]{Black_06}).
In what follows, $\alpha^*$, $\alpha^{**}$ and $\alpha^{**}_{\rm{fin}}$ denote the actions induced by $\alpha$ respectively on $T(A)$,
$B(\partial_eT(A))$ and $A\dfin$.
In order to verify that $\theta$ is equivariant, it is thus sufficient to check
that $\theta_g:= (\alpha_{\mathrm{fin}}^{\ast\ast})_g \circ \theta \circ \alpha^{**}_{g^{-1}}$ also verifies equality \eqref{eq:normal} for every $g \in G$.
To prove this, notice first that, given $\tau \in T(A)$ and $g \in G$, by uniqueness of the boundary measure $\mu_{\tau \circ \alpha_g}$ we have that
\begin{equation} \label{eq:boundary}
\mu_{\tau \circ \alpha_g} = \mu_\tau \circ \alpha^*_g.
\end{equation}
Fix $f \in B(\partial_e T(A))$ and $\tau\in T(A)$.
Then
\begin{align*}
\tau(\theta_g(f)) &= (\tau \circ \alpha_g)(\theta(f \circ \alpha^*_g)) \\
&\stackrel{\mathclap{\eqref{eq:normal}}}{=}  \int_{\partial_e T(A)} f(\alpha_g^*(\lambda)) \, d\mu_{\tau \circ \alpha_g}(\lambda)\\
&\stackrel{\mathclap{\eqref{eq:boundary}}}{=} \int_{\partial_e T(A)} f(\alpha_g^*(\lambda)) \, d\mu_{\tau } (\alpha_g^*(\lambda)) \\
&= \ \int_{\partial_e T(A)} f(\sigma) \, d\mu_{\tau } (\sigma).
\end{align*} 
This proves equivariance of $\theta$. Composing $\theta$ with the unital, equivariant inclusion of $C\left(\overline{\partial_eT(A)}\right)$ into $B(\partial_e T(A))$, one obtains a unital, equivariant  homomorphism $C\left(\overline{\partial_eT(A)}\right)\to Z(A\dfin)$. Since the action on $\overline{\partial_eT(A)}$ is assumed to be amenable, we conclude that the action on $A\dfin$ is von Neumann-amenable (see, for example, \cite[Lemma~3.21]{BusEchWil22}).

$\eqref{item:vN}\Rightarrow \eqref{item:trac amen}$. Fix finite sets $F\subseteq A$ and $K\subseteq G$ and $\varepsilon>0$. We may assume without loss of generality that $F$ consists of contractions.
Throughout the proof, we will identify $A$ with its natural (not necessarily isomorphic) copy inside $A\dfin$. This is possible thanks to \autoref{lma:lifting contractions} and the fact that $\left\|\cdot \right\|_{2,u}$ coincides on $A$ and its copy inside $A\dfin$.

We claim that for every $\delta>0$ and a finite subset $T\subseteq T(A)$, there is $\eta\in C_c(G,A)$ satisfying $\|\eta\|\leq 1$ and
\[1-\|\eta\|_{2,\tau}^2 + \sum_{g\in K}\|\eta-\widetilde{\alpha}_g(\eta)\|_{2,\tau}^2+\sum_{a\in F}\|a\eta-\eta a\|_{2,\tau}^2<\delta,\]
for all $\tau\in T$.

Set $M=A\dfin$ and $\gamma=\alpha\dfin$. Since $\gamma$ is von Neumann
amenable, there exists a net $(\xi_\lambda)_{\lambda\in\Lambda}$ of finitely supported
contractions in $\ell^2(G,Z(M))$ such that 
\[\langle \xi_\lambda, \xi_\lambda\rangle \to 1, \ \ \mbox{ and } \ \ 
 \langle \xi_\lambda-\widetilde{\gamma}_g(\xi_\lambda),\xi_\lambda-\widetilde{\gamma}_g(\xi_\lambda)\rangle \to 0
\]
pointwise on $G$ in the ultraweak topology. Since each trace 
$\tau\in T$ extends to a normal state on $M$, there exists $\lambda\in\Lambda$ such that, with $\xi\coloneqq \xi_\lambda$, we have
\[\tau(1-\langle \xi, \xi\rangle)<\tfrac{\delta}{9} \ \mbox{ and } \ 
 \sum_{g\in K} \tau(\langle \xi-\widetilde{\gamma}_g(\xi),\xi-\widetilde{\gamma}_g(\xi)\rangle)<\tfrac{\delta}{9},
\]
for all $\tau\in T$.
By an application of Kaplansky's density theorem for Hilbert modules (see \cite[Lemma~4.5]{BusEchWil22}), there exists
$\eta\in C_c(G,A)\subseteq \ell^2(G,A)$ with $\|\eta\|\leq 1$ satisfying 
$\tau(\langle \eta-\xi,\eta-\xi\rangle)^{\frac{1}{2}}<\frac{\delta}{9(|F|+|K|)}$, for all $\tau\in T$.
Using the triangle inequality, one gets
\begin{align*}
1-\|\eta\|^2_{2,\tau}  &=\tau(1-\langle \eta, \eta\rangle)<\tfrac{\delta}{3}, \ \ \mbox{ and } \\
 \sum_{g\in K}\|\eta-\widetilde{\alpha}_g(\eta)\|^2_{2,\tau} & =\sum_{g\in K}\tau(\langle \eta-\widetilde{\alpha}_g(\eta),\eta-\widetilde{\alpha}_g(\eta)\rangle)<\tfrac{\delta}{3},
\end{align*}
for all $\tau\in T$.
Since $\xi$ takes values in the center of $M$, then $\eta$ can be chosen so that 
$\sum_{a\in F}\|a\eta-\eta a\|_{2,\tau}^2<\tfrac{\delta}{3}$, for all $\tau\in T$. This proves the claim.

Given $\eta\in C_c(G,A)$, with $\|\eta\|\leq 1$, set
$Q(\eta)\coloneqq \langle \eta,\eta\rangle$ and 
\[D(\eta)\coloneqq 1-Q(\eta)+\sum_{g\in K}Q(\eta-\widetilde{\alpha}_g(\eta))+\sum_{a\in F}Q(a\eta-\eta a).\]
Note that the map $\widehat{D(\eta)}\colon T(A)\to [0,\infty)$
given by
$\widehat{D(\eta)}(\tau)= \tau(D(\eta))$,
for all $\tau\in T(A)$,
is continuous and affine. By the claim above, 
$0$ is in the closure of the set 
\begin{equation*}
Z\coloneqq \left\{\widehat{D(\eta)} \colon \eta\in C_c(G,A), \|\eta\|\leq 1\right\}\subseteq C(T(A))
\end{equation*}
with respect to the topology of pointwise convergence. Since the elements of $Z$ are continuous, affine functions, it follows from \autoref{lma:cts-affine} that $0$ is also in the weak closure of $Z$. Thus, by the Hahn--Banach Theorem, there are $\eta_1,\dotsc,\eta_n\in C_c(G,A)$ with $\|\eta_j\|\leq 1$ for all $j=1,\dotsc,n$, and $\lambda_1,\dotsc,\lambda_n\geq 0$ with $\sum_{j=1}^n \lambda_j =1$, such that 
\begin{equation}\label{eq:convex-combination}
\sum_{j=1}^n \lambda_j \widehat{D(\eta_j)}(\tau)\leq \varepsilon^2
\end{equation}
for all $\tau\in T(A)$. Assuming without loss of generality that $G$ is infinite, and replacing each $\eta_j$ with an appropriate right translate, we can assume that the supports of $\widetilde{\alpha}_g(\eta_i)$ and  $\widetilde{\alpha}_h(\eta_j)$ are disjoint whenever $i\not=j$ and $g,h\in K\cup \{1\}$. Set
\[\xi\coloneqq \sum_{j=1}^n \lambda_j^{\frac 1 2}\eta_j\in C_c(G,A).\]
Using that $\langle \eta_i, \eta_j\rangle \leq \delta_{i,j}$ for $i,j=1,\ldots,n$, one checks that $\|\xi\|\leq 1$, since
\[\|\xi\|^2\leq \sum_{j=1}^n \lambda_j\|\eta_j\|^2\leq 1,\]
and that
$D(\xi)=\sum_{j=1}^n\lambda_j D(\eta_j)$. Hence
\[\sup_{\tau\in T(A)}\tau(D(\xi))=\sup_{\tau\in T(A)}\sum_{j=1}^n \lambda_j \tau(D(\eta_j))\stackrel{\eqref{eq:convex-combination}}{\leq}\varepsilon^2,\]
which implies 
\begin{enumerate}
\item $\left\|\xi a-a \xi\right\|_{2,u} \leq \varepsilon$, for all $a\in F$; 
\item $\left\|\langle \xi,\xi\rangle-1\right\|_{2,u} \leq \varepsilon$; 
\item $\left\|\widetilde{\alpha}_g(\xi)-\xi\right\|_{2,u} \leq \varepsilon$, for all $g\in K$.
\end{enumerate}
This proves \eqref{item:trac amen}. 
\end{proof}

As mentioned before, some of the conditions in \autoref{thm:TracAm} 
resemble similar characterizations of amenability for actions on C$^*$-algebras. For example, the equivalence between \eqref{item:trac amen} and 
\eqref{item:amenable ultrapower} is the tracial version of
the fact that amenability
of an action on $A$ is equivalent to amenability of the induced action on
the norm-central sequence algebra $A_\omega\cap A'$; see
\cite[Theorem~4.4]{Ozawa_Suzuki_2020}, 
while the equivalence between \eqref{item:trac amen} and \eqref{item:vN} is the tracial version of the fact that amenability
of an action on $A$ is equivalent to von Neumann amenability of the induced
action on $A^{\ast\ast}$; see
\cite[Theorem~3.2]{Ozawa_Suzuki_2020} and \cite[Lemma~6.5]{AbadieBussFerraro}. On the other hand, \eqref{item:noncentral}, \eqref{item:T(A)} and \eqref{item:dT(A)}
do not have analogous statements for amenable actions.

The authors would like to thank Siegfried Echterhoff and Rufus Willett for pointing out to them the following observation.

\begin{rem}\label{rem:AmenSA}
The analogue of \eqref{item:trac amen} $\Leftrightarrow$ \eqref{item:T(A)} in \autoref{thm:TracAm}, where one 
replaces tracial amenability with amenability, and $T(A)$ with the state space $S(A)$, is not true. While amenability of an action $\alpha\colon G\to \Aut(A)$ always implies amenability of the induced action on the state space by \cite[Proposition 3.5]{Ozawa_Suzuki_2020}, the converse fails. An easy example can be constructed from any amenable action $G\curvearrowright X$ of a nonamenable group $G$ on a compact space $X$ by considering the associated inner action $\beta$ on $C(X)\rtimes G$. As an inner action, $\beta$ is not amenable as it induces the trivial action on $Z((C(X)\rtimes G)^{**})$. However, $\beta$ induces an amenable action on $S(C(X)\rtimes G)$. Indeed, this follows from \autoref{thm:positive type} and the equivariance of the unital, completely positive map $\iota\colon C(X)\to C(S(C(X)\rtimes G))$ given by $\iota(f)(\varphi)\coloneqq \varphi(f)$, for all $f\in C(X)$ and $\varphi\in S(C(X)\rtimes G)$.
For the same reason, there is no analogue of \eqref{item:trac amen} $\Leftrightarrow$ \eqref{item:noncentral} for amenable actions.

The key tool that allows to obtain these equivalences in the tracial setting is  \cite[Theorem 3]{Ozawa_Dixmier_2013}.
\end{rem}

Given an amenable action $\alpha \colon G \to \text{Aut}(A)$ of a discrete group on a unital C$^*$-algebra, it is not hard to show that 
there exists an invariant state on $A$ if and only if $G$ is amenable. The next lemma, which will be needed in \autoref{S.4}, shows that the tracial counterpart of this statement also holds.

\begin{lma}\label{cor:NoInvariantTraces}
	Let $\alpha\colon G\to \mathrm{Aut}(A)$ be a tracially amenable action of a discrete group on a unital C$^*$-algebra with $T(A)\not= \emptyset$. Then $G$ is amenable if and only if $A$ has an invariant trace. 
\end{lma}
\begin{proof}
If $G$ is amenable, then it is well-known that the affine action
$G\curvearrowright T(A)$ has a fixed point. Conversely, assume
that $\tau\in T(A)$ is $G$-invariant. Then the Dirac measure
on $T(A)$ associated to $\tau$ is also $G$-invariant.
Since $G\curvearrowright T(A)$ is amenable by 
\eqref{item:trac amen}$\Rightarrow$\eqref{item:T(A)} of \autoref{thm:TracAm} (the proof of this implication applies verbatim also if $A$ is not separable), it follows from \cite[Lemma~2.2]{GGKN22}
that $G$ is amenable. 	
\end{proof}

\subsection{Examples} \label{ss.classi}
In this subsection, we use our characterizations of tracial amenability
from \autoref{thm:TracAm} (particularly (7)), 
together with results from the 
classification programme for simple, nuclear C$^*$-algebras, to construct 
a wide class of examples of tracially amenable actions
on unital, simple, stably finite C$^*$-algebras. 
Since we lift automorphisms from the Elliott invariant to the C$^*$-algebra
in question, we work with free groups in order to guarantee that we get 
a group action.
Recall that the \emph{Elliott invariant} of a unital C$^*$-algebra $A$ is 
given by
\[\mathrm{Ell}(A)=((K_0(A),K_0(A)_+,[1_A]), K_1(A), T(A),\rho_A);\]
see \cite[Definition~2.4]{GLN_classification_2015} for its definition.

\begin{eg}\label{eg:TracAmenEll}
Let $n\in \N$, let $F_n=\langle g_1,\ldots,g_n\rangle$ be the free group on $n$ generators, and let $\theta\colon F_n\acton X$ be an amenable action on a compact, metric space (for example, let $X$ be the Gromov boundary $\partial F_n$). By \cite[Theorem~14.8]{GLN_classification_2015} (see also \cite[Theorem~2.8]{EGLN_classification_2016}, or \cite{new_classification}), there is a stably finite, classifiable C$^*$-algebra $A$ with
\[\mathrm{Ell}(A)\cong 
((\mathbb{Z},\mathbb{Z}_+,1),\{0\},\mathrm{Prob}(X),\rho_A),\]
where $\rho_A\colon K_0(A)\times T(A)\to \mathbb{R}$ is the (uniquely determined) pairing map given by $\rho_A(n[1_A],\tau)=n$ for all 
$n\in\Z$ and all $\tau\in T(A)$.
Let $\widehat{\theta}\colon F_n\acton \mathrm{Prob}(X)$ be the action induced by $\theta$. For each $j=1,\ldots,n$, the triple 
$(\mathrm{id}_{K_0(A)}, \mathrm{id}_{K_1(A)}, \widehat{\theta}_{g_j})$ is an automorphism of $\mathrm{Ell}(A)$ in the sense of 
\cite[Definition~2.4]{GLN_classification_2015}, and thus by \cite[Theorem~29.8]{GLN_classification_2015} (see also \cite[Theorem~2.7]{EGLN_classification_2016}, or \cite{new_classification}) there exists an automorphism $\alpha_j\in \mathrm{Aut}(A)$ such that $\mathrm{Ell}(\alpha_j)=(\mathrm{id}_{K_0(A)}, \mathrm{id}_{K_1(A)}, \widehat{\theta}_{g_j})$. Using the universal property of $F_n$, we obtain an action $\alpha\colon F_n\to \mathrm{Aut}(A)$ whose induced action on $\partial_eT(A)\cong X$
is conjugate to $\theta$. Tracial amenability for $\alpha$ then follows directly from the equivalence between 
\eqref{item:trac amen} and \eqref{item:dT(A)}
in \autoref{thm:TracAm}.
\end{eg}

In the above construction there is a significant 
amount of freedom in the choice of the Elliott invariant of the C$^*$-algebra $A$. 
The following variant of the construction above was suggested
to the authors by Mikael R{\o}r\-dam, and it has the advantage of 
readily producing actions on unital, simple AF-algebras.

\begin{eg}\label{prop:ExOnAFAlg}
Let $\theta \colon F_n \curvearrowright X$ be an amenable action on the
Cantor space $X$.
By \cite[Proposition~1.4.5]{Ror_classification_2002}, there is a unital AF-algebra $A$
with
\[(K_0(A),K_0(A)_+,[1])\cong (C(X,\Q^{\mathrm{d}}),C(X,\Q^{\mathrm{d}})_{++}\cup \{0\},1),\]
where $\Q^{\mathrm{d}}$ denotes the rational numbers with the discrete topology, and $C(X,\Q^{\mathrm{d}})_{++}$ denotes the strictly positive continuous functions $X\to \Q^\mathrm{d}$. Since
$(K_0(A),K_0(A)_+)$ is
simple in the sense of \cite[Definition 1.5.1]{Ror_classification_2002},
it follows from 
\cite[Corollary~1.5.4]{Ror_classification_2002} that
$A$ is simple. 

Use Elliott's classification of AF-algebras (in the form stated
in \cite[Theorem~1.3.3~(ii)]{Ror_classification_2002}) to lift the automorphisms of $(K_0(A),K_0(A)_+,[1])$ induced by the 
homeomorphisms $\theta_{g_1},\dotsc,\theta_{g_n}$ to automorphisms $\alpha_{g_1},\dotsc,\alpha_{g_n}$ of $A$. We again obtain an action $\alpha \colon F_n \to \text{Aut}(A)$. 
Since there is a natural continuous affine homeomorphism $T(A)\cong \Prob(X)$ by \cite[Proposition~1.5.5]{Ror_classification_2002},
it follows that the action that $\alpha$ induces on $T(A)$ is
conjugate to $\theta$. Hence $\alpha$ is tracially amenable by \autoref{thm:TracAm}.
\end{eg}

\section{Purely infinite crossed products} \label{S.4}
In this section, we investigate the internal structure of reduced crossed products by tracially amenable actions. For a large class of nonamenable groups, we show that outer, tracially amenable 
actions on simple, unital, $\mathcal{Z}$-stable C$^*$-algebras 
produce simple, purely infinite reduced crossed products.
Our proof proceeds in two steps. The first one is 
\autoref{prop:CompPurInf}, where we show in great generality that dynamical strict comparison is sufficient to 
conclude pure infiniteness of the reduced crossed product.
The second step is
\autoref{thm:Mainthm}, the main result of this work, where we show that
tracial amenability implies dynamical strict comparison for a vast class of actions of nonamenable groups on simple C$^*$-algebras. 

\subsection{From dynamical comparison to pure infiniteness of crossed
products}
We begin by recalling some preliminaries about Cuntz comparison,
and refer the reader to \cite{AntPerThi_tensor_2018, ThielCuntzsemigrp, GP} 
for modern introductions to the subject.

Let $A$ be a C$^*$-algebra. 
For positive elements $a,b\in A_+$, we say that $a$ is 
\emph{Cuntz subequivalent} to $b$, written $a\precsim b$, 
if there exists a sequence $(r_n)_{n\in \N}$ in $A$ such that $\lim_{n \to \infty} \|r_nbr_n^*-a\|= 0$. 

For a unital C$^*$-algebra $A$, we denote by $QT(A)$ the set of (everywhere
defined) normalized (2-)quasitraces on $A$
(\cite[Definition 6.7]{ThielCuntzsemigrp}). We will be mainly interested in C$^*$-algebras where $QT(A) = T(A)$, which is the case, for instance, whenever $A$ is exact (\cite{Haa_quasitraces_2014}).

Given $\tau \in QT(A)$, there exists a unique extension of $\tau$
to a lower semicontinuous (unbounded) quasitrace on $A \otimes \mathcal{K}$. With a slight abuse of notation, we denote such extension again
by $\tau$.
For $a \in (A \otimes \mathcal{K})_+$ and $\tau \in QT(A)$, set
\[
d_\tau(a)\coloneqq \lim_{n\to\infty} \tau(a^{1/n})\in [0,\infty].
\]

The following lemma, which was pointed out to us by Sam Evington, will be needed in the proof of \autoref{thm:Mainthm}.
For $\ep>0$, we denote by $f_\ep\colon [0,\infty)\to [0,\infty)$ the function given by $f_\ep(t)=\max \{ 0,t-\ep \}$ for $t\in [0,\infty)$.
For $a\in A_+$, we will usually write $(a-\ep)_+$ for $f_\ep(a)$.

\begin{lma}\label{lma:Sam}
Let $A$ be a unital C$^*$-algebra, and let $\gamma,\ep>0$.
\begin{enumerate}
 \item\label{clm:continuity} For every $f\in C([0,1])$ there exists $\delta >0$ such that for all contractions $a ,b \in A_+$ and all 
 $\tau \in T(A)$ satisfying $\left\| a - b \right\|_{2,\tau} < \delta$, we have
\[
	 \left\| f(a) - f(b) \right\|_{2,\tau} < \gamma.
\]
\item\label{item:Sam2} There is $\delta>0$ such that 
for all contractions $a ,b \in A_+$ and all 
$\tau \in T(A)$ satisfying $\left\| a - b \right\|_{2,\tau} < \delta$, we have 
\[d_\tau((a-\ep)_+)<d_\tau(b)+\gamma.\] 
\end{enumerate}
\end{lma}
\begin{proof}
(1). Suppose that there is $n\in\N$ such that $f(t)=t^n$ for all 
$t\in [0,1]$. Then
\begin{align*}
\|f(a)-f(b)\|_{2,\tau}&\leq \|a-b\|_{2,\tau}\sum_{j=0}^{n-1}\|a\|^j\|b\|^{n-1-j}
\leq n\|a-b\|_{2,\tau}.
\end{align*}
Thus, in this case we may take $\delta=\tfrac{\gamma}{n}$. 
The case where $f$ is a polynomial can be deduced from this, and the
general case follows using a density argument.
\vspace{.2cm}

(2). Let $g_\ep\in C([0,1])$ be the function defined as $g_\ep(t)\coloneqq\min\{\ep^{-1} t,1\}$, for all $t\in [0,1]$. 
Using part~(1), find $\delta>0$ such that whenever $a,b\in A_+$ are contractions satisfying $\|a-b\|_{2,\tau}<\delta$, then $\|g_\ep(a)-g_\ep(b)\|_{2,\tau}<\gamma$. By part~(1) of \autoref{lma:Cauchy-Schwarz}, we have
\begin{equation}\label{eqn:trIneq}
|\tau(g_\ep(a)-g_\ep(b))|\leq \|g_\ep(a)-g_\ep(b)\|_{2,\tau}<\gamma.
\end{equation}
Denote by $\mu$ the measure on the spectrum of $a$ induced by the trace $\tau|_{C^*(1,a)}$, so that
\begin{equation}\label{eqn:dtau}
d_\tau((a-\ep)_+)= \mu\big(\{x\in \mathrm{sp}(a)\colon f_\ep(x)>0\}\big) = \mu((\ep,1]).
\end{equation}
Using that $g_\ep=1$ on $(\ep,1]$ at the second step, we get 
\[d_\tau((a-\ep)_+)\stackrel{\eqref{eqn:dtau}}{=}\mu((\ep,1])\leq \int_{[0,1]} g_\ep d\mu=\tau(g_\ep(a))\stackrel{\eqref{eqn:trIneq}}{\leq} \tau(g_\ep(b))+\gamma \leq d_\tau(b)+\gamma,\] 
as desired.
\end{proof}

For a compact convex set $K$, we denote by $\mathrm{LAff}(K)_+$ (respectively, $\mathrm{LAff}(K)_{++}$) the set of lower semicontinuous positive (respectively, strictly positive) affine functions on $K$.

\begin{df}\label{df:SurRk}
Let $A$ be a unital C$^*$-algebra. For $a\in (A\otimes \mathcal{K})_+$, we define its \textit{rank function} to  be the map $\mathrm{rk}(a)\colon QT(A)\to [0,\infty]$ given by
\[\mathrm{rk}(a)(\tau)=d_\tau(a)\]
for all $\tau\in QT(A)$. 

Given a simple, unital $C^*$-algebra $A$ and $a\in(A\otimes \mathcal{K})_+\setminus\{0\}$, it is always true that $d_\tau(a)>0$, for every $\tau\in QT(A)$. As $\mathrm{rk}(a)$ is a lower semicontinuous map on a compact set, it attains its minimum. It follows that
$\mathrm{rk}(a)\in \mathrm{LAff}(QT(A))_{++}$, for every $a\in(A\otimes \mathcal{K})_+\setminus\{0\}$. We say that \emph{all ranks are realized} in $A$ if the rank map $\mathrm{rk}\colon (A\otimes \mathcal{K})_+\setminus\{0\}\to \mathrm{LAff}(QT(A))_{++}$ is surjective.
\end{df}

One of the assumptions of \autoref{thm:Mainthm} is that all ranks are
realized for the coefficient C$^*$-algebra. This is in fact a very mild
assumption, which is automatic in many cases
of interest. This is for example the case for nonelementary, simple, separable, unital
C$^*$-algebras of stable rank one by \cite[Theorem~8.11]{ThielSR1}.
Moreover, it is an open question of Nate Brown whether all ranks 
are realized in every nonelementary, unital, simple, separable, 
stably finite C$^*$-algebra (see \cite[Question~1.1]{ThielSR1}). 

We will need the notion of strict comparison. Although this is most commonly
defined using tracial states, we choose to do it using normalized quasitraces 
since this version is more compatible with the Cuntz semigroup outside of the 
exact setting.

\begin{df}
A unital, simple C$^*$-algebra $A$ is said to have 
\textit{strict comparison} if we have $a\precsim b$ whenever $a,b\in (A\otimes \mathcal K)_+\setminus\{0\}$ 
satisfy $\mathrm{rk}(a)<\mathrm{rk}(b)$. 
\end{df}

\begin{rem}\label{rem:StrCompPI}
When $A$ is a simple, unital C$^*$-algebra, and $QT(A)=\emptyset$, then strict comparison is equivalent to the fact that $1\precsim a$ for all nonzero $a\in (A\otimes\K)_+$. This, in turn, is equivalent to pure infiniteness of $A$ (see \cite[Definition~4.1]{KirRor_nonsimple_2000} and \cite{Cun_1981}).
\end{rem}

The following definition originates from an ongoing project by Bosa, Perera, Wu and Zacharias,
and it provides a noncommutative generalization of the notion of dynamical subequivalence from \cite[Definition~3.2]{Ker_dimension_2020}.

\begin{df}\label{df:DynSubEq}
Let $\alpha\colon G\to \mathrm{Aut}(A)$ be an action of a countable, discrete group $G$ on a C$^*$-algebra $A$, and let $a,b\in A_+$.
\begin{enumerate}
\item Write $a\precsim_0 b$ if for any $\ep>0$ there are $\delta>0$, $n\in\mathbb{N}$, $g_1,\ldots,g_n\in G$, and positive elements $x_1,\ldots, x_n\in M_\infty(A)$, such that 
\[(a-\ep)_+\precsim \oplus_{j=1}^{n} \alpha_{g_j}(x_j) \quad \text{ and } \quad \oplus_{j=1}^{n}x_j\precsim (b-\delta)_+.\] 

\item Write $a\precsim_G b$ if there are $x_1,\ldots,x_n\in M_\infty(A)_+$ with
\[a=x_1\precsim_0 x_2\precsim_0 \ldots \precsim_0 x_n=b. \] 
\end{enumerate}
\end{df}

The need for (2) stems for the fact that, unlike for $\precsim$, 
it is not clear whether the relation $\precsim_0$ is transitive.

\begin{rem}\label{rem: CompNoepsDel}
	Assume that there are $n\in\mathbb{N}$, elements $g_1,\ldots,g_n\in G$, and positive elements $x_1,\ldots, x_n\in M_\infty(A)$ such that 
	\[a\precsim \oplus_{j=1}^{n} \alpha_{g_j}(x_j) \quad \text{ and } \quad \oplus_{j=1}^{n}x_j\precsim b.\] 
	Then $a\precsim_0 b$. This follows from double application of \cite[Proposition~2.4]{Ror_1992}.
\end{rem}

\begin{rem}\label{rem: DyncomVScom}
	One can check that $a\precsim_G b$ entails $a\precsim b$ in $A\rtimes_r G$, by putting together \cite[Proposition~2.4]{Ror_1992}, \cite[Proposition~2.18]{ThielCuntzsemigrp}, and the fact that the action on $A$ is unitarily implemented inside $A\rtimes_r G$.
\end{rem}

The notion of dynamical strict comparison below 
is due to Bosa, Perera, Wu and Zacharias, as part of their ongoing
work on almost elementariness extending 
the ideas and techniques developed in \cite{Ker_dimension_2020}
to the noncommutative setting.

\begin{df}\label{df:DynCom}
An action $\alpha\colon G\to \mathrm{Aut}(A)$ of a countable, discrete group $G$ on a unital C$^*$-algebra $A$ is said to have \textit{dynamical strict comparison} if we have $a\precsim_G b$ for all $a,b\in (A\otimes\K)_+\setminus\{0\}$ 
satisfying $d_\tau(a)<d_\tau(b)$ for all invariant quasitraces $\tau \in QT(A)$. 
\end{df}

In the definition above, we deliberately allow for the possibility that there are no
invariant quasitraces, in which case 
dynamical strict comparison is equivalent to the 
statement that
$a\precsim_G b$ for all nonzero $a,b\in (A\otimes\mathcal{K})_+$. 

Note that when $\alpha$ is trivial, 
dynamical strict comparison for $\alpha$
is equivalent to strict comparison for $A$. 
The following is a generalization of \cite[Theorem~1.1]{Ma_comparison_2019} to the noncommutative setting (see also \cite[Theorem~2.8]{GGKN22}). Recall that an action
$\alpha\colon G\to\Aut(A)$ is said to be \emph{outer} if
$\alpha_g$ is not inner for all $g\in G\setminus\{1\}$.

\begin{prop}\label{prop:CompPurInf}
Let $\alpha\colon G\to \mathrm{Aut}(A)$ be an outer action of a 
countable, discrete group $G$ on a unital, simple C$^*$-algebra $A$. 
Assume that $\alpha$ has dynamical strict comparison and no 
invariant quasitraces. Then $A\rtimes_r G$ is simple and purely 
infinite.
\end{prop}
\begin{proof}
Simplicity follows from \cite[Theorem~3.1]{Kishimoto}. 
Let $a\in A\rtimes_r G$ be a nonzero positive element. We will show that $1 \precsim a$ in $A\rtimes_r G$, which implies that $A\rtimes_r G$ is purely infinite by \autoref{rem:StrCompPI}. 
By \cite[Lemma~3.2]{RordamSierakowski} (see also \cite[Lemma~4.2]{ESPA})
there exists a nonzero element $b\in A_+$ such that $b\precsim a$. 
In the absence of invariant quasitraces, dynamical strict comparison gives $1\precsim_G b$. By \autoref{rem: DyncomVScom}, we have that $1\precsim b\precsim a$ in $A\rtimes_r G$, as desired.
\end{proof}

\subsection{Groups with weak paradoxical towers}
Next, we introduce the notion of weak paradoxical towers, based on ideas from \cite{GGKN22}. This is a technical assumption
which allows us to link tracial amenability to strict comparison. The class of groups having weak paradoxical towers is
quite large, and it for example includes
all groups containing the free group $F_2$; see \autoref{cor:GrpsContF2}.

\begin{df}\label{df:WeakParTwrs}
	Let $n\in \N$. A group $G$ is said to have \emph{weak $n$-paradoxical towers} if for every $m\in \N$, there exist a finite subset $D\subseteq G$ with $|D|\geq m$, subsets $K_1,\dotsc K_n\subseteq G$, and group elements $g_1,\dotsc,g_n\in G$ such that
	\begin{enumerate}
		\item For each $j=1,\dotsc,n$, the sets $\{dK_j\}_{d\in D}$ are pairwise disjoint;
		\item $\bigcup_{j=1}^n g_j K_j=G$.
	\end{enumerate}
	A group is said to have \emph{weak paradoxical towers} if it has weak $n$-paradoxical towers for some $n\in \N$. 
\end{df}

We make the following observation:
\begin{lma}\label{lma:WkParNonAm}
Let $G$ be a group with weak paradoxical towers. Then $G$ is not amenable. 
\end{lma}

Using elementary arguments, one can show that all nonabelian free
groups have weak 2-paradoxical towers; see 
\cite[Proposition~3.2]{GGKN22} for a stronger result. 
More generally, groups that have paradoxical towers in the sense of 
\cite[Definition~3.1]{GGKN22} have weak paradoxical towers in our
sense, and in particular all the groups considered in 
\cite[Section 4]{GGKN22} have weak paradoxical towers. The class 
of groups we consider here is really larger than the one 
considered in \cite{GGKN22}: for example 
$F_2\times \Z$ does not have paradoxical towers by 
\cite[Remark~3.7]{GGKN22}, but it has \emph{weak} paradoxical
towers by \autoref{cor:GrpsContF2} below. 
The key difference is that weak paradoxical
towers can be induced from subgroups:

\begin{prop}\label{prop:SubgrpWeakPar}
	Let $G$ be a group and $H\subseteq G$ a subgroup with weak $n$-paradoxical towers for some $n\in \N$. Then $G$ has weak $n$-paradoxical towers.
\end{prop}
\begin{proof}
	Let $m\in \N$. Since $H$ has weak $n$-paradoxical towers, there are a finite set $D\subseteq H$ with $|D|\geq m$, subsets $L_1,\dotsc,L_n\subseteq H$, and elements $g_1,\dotsc,g_n\in H$, such that the sets $\{d L_j\}_{d\in D}$ are pairwise disjoint for every $j=1,\dotsc,n$, and such that $H=\bigcup_{j=1}^ng_j L_j$. Let $S\subseteq G$ be a set satisfying $G=\bigsqcup_{s\in S}Hs$, and set $K_j\coloneqq \bigsqcup_{s\in S}L_js,~j=1,\dotsc,n$. One readily checks that $\bigcup_{j=1}^ng_jK_j=G$, and that the sets $\{dK_j\}_{d\in D}$ are pairwise disjoint, for every $j=1,\ldots,n$.
\end{proof}

A combination of \autoref{prop:SubgrpWeakPar} with \cite[Proposition~3.2]{GGKN22} gives the following corollary.

\begin{cor}\label{cor:GrpsContF2}
Every group containing $F_2$ has 
weak paradoxical towers. 
\end{cor}

\subsection{Establishing dynamical comparison from tracial
amenability}
We are now ready to prove the main result of this section.
As before, given $\ep>0$ we denote by $f_\ep\colon [0,\infty)\to [0,\infty)$ the function $f_\ep(t)\coloneqq \max\{0,t-\ep\}$ for $t>0$, and for a positive element $a$ we write $(a-\ep)_+$ for 
$f_\ep(a)$. Since functional calculus commutes with automorphisms,
for $\varphi\in\Aut(A)$ we have $\varphi((a-\ep)_+)=(\varphi(a)-\ep)_+$.

\begin{thm}\label{thm:Mainthm}
	Let $G$ be a countable, discrete group with weak paradoxical towers, let $A$ be a unital, simple, separable C$^*$-algebra with $QT(A)=T(A)\neq \emptyset$, with strict comparison, and for which all ranks are realized. Then, any tracially amenable action $\alpha: G\to\mathrm{Aut}(A)$ has dynamical strict comparison.
\end{thm}

\begin{proof}
Since $G$ has weak paradoxical towers, it is not amenable 
by \autoref{lma:WkParNonAm} and thus has no invariant traces by \autoref{cor:NoInvariantTraces}. In particular, 
dynamical strict comparison for $\alpha$ amounts to showing that 
$1\precsim_G a$ for every nonzero $a\in A_+$. Fix such $a$. 
Since $\mathrm{rk}(a)$ is a strictly positive, lower semicontinuous function on the compact space $T(A)$, it attains its infimum
and thus there is $m\in \N$ such that 
\begin{equation}\label{eq:ma>1}
m\cdot\mathrm{rk}(a)\geq 1.
\end{equation}
Let $n\in\mathbb{N}$ be such that $G$ has weak $n$-paradoxical towers. Fix a finite subset $D\subseteq G$ with $\frac{3n}{|D|-1}(2n^2+1)<\frac 1 m$,
and use weak paradoxical towers for $G$ to 
find $K_1,\dotsc,K_n\subseteq G$ and $g_1,\dotsc,g_n\in G$ such that
\begin{enumerate}[label=(a\arabic*)]
\item $\{d K_j\}_{d\in D}$ are pairwise disjoint for all $j=1,\dotsc,n$; \label{eq:Aj disjoint}
\item $\bigcup_{j=1}^n g_jK_j=G$. \label{eq:Aj cover}
\end{enumerate}
Let $0<\ep<\frac 1{18n|D|}$. Apply part~(2) of \autoref{lma:Sam} to find $\delta>0$ such that for all positive contractions $s,t\in A$ satisfying $\|s-t\|_{2,u}<\delta$, we have 		
\begin{equation}\label{eq:delta}
d_\tau((s-\ep)_+)<d_\tau(t)+\ep, 
\end{equation}	
for all $\tau\in T(A)$.
Apply now part~(1) of \autoref{lma:Sam}, find $0<\gamma\leq \ep$, such that for all positive contractions $s,t\in A$ satisfying $\|s-t\|_{2,u}<\gamma$, we have 
\begin{equation}\label{eq:gamma}
 \left\|\big(s-\tfrac 1 {2n}\big)_+-\big(t-\tfrac 1 {2n}\big)_+\right\|_{2,u}<\delta.
\end{equation}	

By tracial amenability of $\alpha$, there is $\xi\in C_c(G,A)$ satisfying 
\begin{enumerate}[label=(b\arabic*)]
\item $\|\xi\|\leq 1$; \label{eq:xi contractive}
\item $\|\langle \xi,\xi\rangle -1\|_{2,u}<\ep$; \label{eq:xi unital}
\item $\|\widetilde{\alpha}_g(\xi)-\xi\|_{2,u}<\frac \gamma 2$, for all $g\in D\cup\{g_1,\dotsc,g_n\}$. \label{eq:xi invariant}
\end{enumerate} 
For $j=1,\dotsc,n$, set 
\[c_j=\sum_{g\in K_j}\xi(g)^*\xi(g), \ \ b_j= \big(c_j-\tfrac 1 {2n}\big)_+, \ \
 \mbox{ and } \ \ b\coloneqq b_1\oplus\dotsb\oplus b_n.
\]
For an element $x\in A$ and $\ell\in\mathbb{N}$, we will denote by $x^{\oplus \ell}\in M_\ell(A)$ the $\ell$-fold direct sum of $x$ with itself.
We will prove that $1\precsim_G a$ in two steps.

\begin{clm}\label{clm:1<b}
We have $1\precsim_0 (b-\ep)_+^{\oplus 3n}$.
		 \end{clm}
		 
		 For $j=1,\dotsc,n$ and $h\in D\cup \{g_1,\dotsc,g_n\}$, we denote by $1_{hK_j}\in \mathcal L_A(\ell^2(G,A))$ the projection onto the elements supported on $hK_j\subseteq G$. Using the Cauchy-Schwarz inequality (see part~(2) of \autoref{lma:Cauchy-Schwarz}) for the third step, we obtain 
			\begin{align}\label{eq:invTowers1}		 	
		 			 		\Big\|\alpha_h(c_j)&-\sum_{g\in hK_j}\xi(g)^*\xi(g)\Big\|_{2,u}\\
		 			 		&=\left\|\langle 1_{hK_j}\widetilde{\alpha}_h(\xi),\widetilde{\alpha}_h(\xi)\rangle -\langle 1_{hK_j}\xi,\xi\rangle \right\|_{2,u}\nonumber\\
		 		&\leq \left\| \langle 1_{hK_j}\widetilde{\alpha}_h(\xi),\widetilde{\alpha}_h(\xi)-\xi\rangle \right\|_{2,u}+ \left\|\langle \xi-\widetilde{\alpha}_h(\xi),1_{hK_j}\xi\rangle \right\|_{2,u} \nonumber \\
&\leq 2\left\|\widetilde{\alpha}_h(\xi)-\xi\right\|_{2,u}\nonumber  <\gamma \nonumber.
\end{align}

Combining \eqref{eq:gamma} and \eqref{eq:invTowers1}, we get
\begin{equation}\label{eq:invTowers2}
\Big\|\alpha_h(b_j)-\Big(\sum_{g\in hK_j}\xi(g)^*\xi(g)-\tfrac{1}{2n}\Big)_+\Big\|_{2,u}<\delta,
\end{equation}
for all $h\in D\cup \{g_1,\ldots,g_n\}$ and $j=1,\ldots, n$.
Fix $\tau\in T(A)$. For $j\in\{1,\dotsc,n\}$, denote by $\mu_j$ the measure on the spectrum of $\alpha_{g_j}(c_j)$ corresponding to $\tau|_{C^*\left(1,\alpha_{g_j}(c_j)\right)}$. We have
\begin{align}\label{eq:dtau1}
d_\tau(\alpha_{g_j}((b_j-\ep)_+))&=d_{\tau}\big(\big(\alpha_{g_j}(c_j)-\ep-\tfrac 1 {2n}\big)_+\big)\\ &=\int_{\big(\ep+\tfrac 1 {2n},1\big]}1d\mu_j\nonumber \\
&\geq \int_{[0,1]}td\mu_j-\big(\ep+\tfrac 1 {2n}\big) \nonumber \\
&= \tau(\alpha_{g_j}(c_j))-\big(\ep+\tfrac 1 {2n}\big)\nonumber \\
&\overset{\eqref{eq:invTowers1}}\geq \tau\Big(\sum_{g\in g_jK_j}\xi(g)^*\xi(g)\Big)-\big(\ep+\tfrac 1 {2n}+\gamma\big). \nonumber
\end{align}		
Furthermore, we have
\begin{align*}
\sum_{j=1}^n \tau \Big( \sum_{g\in g_jK_j}\xi(g)^*\xi(g)\Big)  &\stackrel{\mathclap{\ref{eq:Aj cover}}}{\ge} \tau \Big( \sum_{g \in G} \xi(g)^*\xi(g) \Big) \\
& \stackrel{\mathclap{\ref{eq:xi unital}}}{\ge} 1 -\varepsilon.
\end{align*}
		 	We conclude that there is $j\in \{1,\dotsc,n\}$ for which 
		 	\begin{equation}\label{eq:dtau2}
		 		\tau\Big(\sum_{g\in g_jK_j}\xi(g)^*\xi(g)\Big)\geq  \frac{1 - \varepsilon}{n} \ge \frac{1}{n} - \varepsilon.
		 	\end{equation}
		 	Combining \eqref{eq:dtau1} and \eqref{eq:dtau2}, we obtain
		 	\begin{align*}
			\sum_{j=1}^nd_\tau(\alpha_{g_j}((b_j-\ep)_+))\geq \tfrac 1 n-\ep-\Big(\ep+\tfrac 1 {2n}+\gamma\Big)\geq \tfrac 1 {2n} -3\ep 
			>\tfrac 1 {3n},
			\end{align*}
		 	for all $\tau\in T(A)$. Using strict comparison for $A$, this implies that
		 		\begin{equation}\label{eq:1<b}
		 			1\precsim \oplus_{j=1}^n\alpha_{g_j}((b_j-\ep)_+)^{\oplus 3n}.
		 		\end{equation}
		 	As $\oplus_{j=1}^{n}(b_j-\ep)_+^{\oplus 3n}=(b-\ep)_+^{\oplus 3n}$, it follows that $1\precsim_0 (b-\ep)_+^{\oplus 3n}$, proving \autoref{clm:1<b}.
		 	
		 	\begin{clm}\label{clm:b<a}
		 		We have $(b-\ep)_+^{\oplus 3n}\precsim_0 a$.
		 	\end{clm}		 	
		 	Fix $\tau\in T(A)$, $d\in D$, and $j\in \{1,\dotsc,n\}$. Denote by $\mu$ the probability measure on the spectrum of $\sum_{g\in dK_j}\xi(g)^*\xi(g)$ corresponding to the restriction of $\tau$ to $C^*\big(1,\sum_{g\in dK_j}\xi(g)^*\xi(g)\big)$. Then
\begin{align}\label{eq:dtau3}
d_\tau(\alpha_d((b_j-\ep)_+)) \quad &\stackrel{\mathclap{\eqref{eq:delta}, \eqref{eq:invTowers2}}}{\leq} \quad d_\tau\Big(\Big(\sum_{g\in dK_j}\xi(g)^*\xi(g)-\tfrac 1 {2n}\Big)_+\Big) +\ep \\
&=\int_{\left(\frac 1 {2n},1\right]}1d\mu+\ep\nonumber\\
&\leq 2n\int_{[0,1]}td\mu+\ep \nonumber\\
&= 2n \tau\Big(\sum_{g\in dK_j}\xi(g)^*\xi(g)\Big)+\ep. \nonumber
\end{align}
Taking the sum of \eqref{eq:dtau3} over $j\in \{1,\dotsc,n\}$ and $d\in D$ yields			
\begin{align}\label{eq:dtau summed up}
\sum_{d\in D}\sum_{j=1}^{n}d_\tau(\alpha_d((b_j-\ep)_+))&\leq 2n\sum_{j=1}^{n}\tau\Big(\sum\limits_{d\in D}\sum_{g\in dK_j}\xi(g)^*\xi(g)\Big)+n|D|\ep\\
				&\stackrel{\mathclap{\ref{eq:Aj disjoint},\ref{eq:xi contractive}}}{\le} \quad
				2n^2+1. \nonumber
				\end{align}		 	
		 	Since all ranks are realized in $A$, we can find positive elements $b'_d\in M_\infty(A)$ for $d\in D$ satisfying 
		 		\begin{equation}\label{eq:betad}
		 			\mathrm{rk}(b'_d)=\frac{3n}{|D|-1}\mathrm{rk}(\alpha_d((b-\ep)_+)).
		 		\end{equation}		 		
		 		Note that $(b-\ep)_+\not=0$ by \autoref{clm:1<b}. Using this at the second step, we get
		 		\begin{equation*}
		 		\begin{aligned}
					3n\cdot d_\tau((b-\ep)_+)&=\frac{3n}{|D|}\sum_{d\in D}d_{\tau\circ \alpha_{d^{-1}}}(\alpha_d((b-\ep)_+))\\
					&<\frac{3n}{|D|-1}\sum_{d\in D}d_{\tau\circ\alpha_{d^{-1}}}(\alpha_d((b-\ep)_+))\\
					&\overset{\eqref{eq:betad}}=\sum_{d\in D}d_\tau(\alpha_{d^{-1}}(b'_d)),
				\end{aligned}
	 			\end{equation*}
	 			for all $d\in D$ and $\tau\in T(A)$. Using strict comparison for $A$, it follows that
	 			\begin{equation}\label{eq:b<beta}
	 			 (b-\ep)_+^{\oplus 3n}\precsim \oplus_{d\in D}\alpha_{d^{-1}}(b'_d).
	 			 \end{equation}
	 			On the other hand, we have 
	 			\begin{equation*}
	 				\begin{aligned}
	 					 \sum_{d\in D}d_\tau(b'_d)&\overset{\eqref{eq:betad}}=\frac{3n}{|D|-1}\sum_{d\in D}d_\tau(\alpha_d((b-\ep)_+))\\&\overset{\eqref{eq:dtau summed up}}\leq \frac{3n}{|D|-1}(2n^2+1)\\ &\overset{\eqref{eq:ma>1}}<d_\tau(a)
	 				\end{aligned}
	 			\end{equation*}
	 			for all $\tau\in T(A)$. 
	 			Again by strict comparison, this implies that 
 				\begin{equation}\label{eq:beta<a}
					\oplus_{d\in D}b'_d\precsim a,			
 				\end{equation}
				and therefore
				\[
				(b - \varepsilon)_+^{\oplus 3n} \precsim_0 a.
				\]
				
 				Putting together both claims, we get $1\precsim_G a$, as required.
		 \end{proof}

The following is the desired result for crossed products. 

\begin{cor}\label{cor:main}
Let $G$ be a countable, discrete group with weak paradoxical towers, let $A$ be a unital, simple, separable C$^*$-algebra with $QT(A)=T(A)\not =\emptyset$, with strict comparison and for which all ranks are realized, and let $\alpha\colon G\to \mathrm{Aut}(A)$ be a tracially amenable, outer action. Then $A\rtimes_r G$ is simple and purely infinite. 
\end{cor}
\begin{proof} 
The statement follows from \autoref{thm:Mainthm} and \autoref{prop:CompPurInf}.
\end{proof}

Let $A$ be a simple, separable, exact, unital, stably finite 
C$^*$-algebra. If $A$ has strict 
comparison and stable rank one, then all ranks are realized on it 
by \cite[Theorem~8.11]{ThielSR1}, and so $A$ satisfies the assumptions of \autoref{cor:main}. In particular, this is the
case if $A$ is $\mathcal{Z}$-stable, by  \cite[Theorem~4.5 and Theorem~6.7]{Ror_2004}. Instead of stable rank one, one can equivalently require $A$ to be almost divisible. 
Indeed, simple, separable, unital, stably finite C$^*$-algebras with strict comparison and almost divisibility have stable rank one by work in progress of Geffen and Winter (see also \cite{LinSR1}, where the assumption of almost divisibility is replaced by tracial approximate oscillation zero).

We close the paper with an application of our results to classifiable C$^*$-algebras. In the following corollary, we require
amenability of the action, as tracial amenability is not sufficient to grant nuclearity of the crossed product.

\begin{cor}\label{cor:uct} 
Let $G$ be a countable, discrete group with weak paradoxical towers, 
and let $A$ be a simple, separable, unital,
nuclear, stably finite, $\mathcal{Z}$-stable C$^*$-algebra. Suppose that
$\alpha\colon G\to \Aut(A)$ is an amenable, outer action\footnote{We do not actually know if such an action exists; see Problem~D in the introduction.}. Then $A\rtimes_r G$ is a unital Kirchberg algebra. If $A$
satisfies the UCT and $G$ is torsion-free
and 
satisfies the Haagerup property (for example, $G=F_n$; see \cite[Definition 12.2.1]{BrownOzawa}), then $A\rtimes_rG$ also
satisfies the UCT, and is therefore completely determined by its 
$K$-theory.
 \end{cor}
 \begin{proof} 
 By \autoref{cor:main} and the comments after it, 
 $A\rtimes_r G$ is simple, separable, unital and purely infinite. 
 By \cite[Theorem~4.5]{AD_systemes_1987}, it is also nuclear. 
 The fact that it satisfies the UCT follows from \cite{HigsonKasparov} 
 and \cite[Corollary~9.4]{MeyerNest} (see \cite[Corollary 7.2]{RosSch_kunneth_1987} for the case $G = F_n$). 
 \end{proof}

We remark that \autoref{cor:uct} holds also when $A$ is purely infinite, by applying \cite[Lemma~3.2]{RordamSierakowski} (see also \cite[Lemma~4.2]{ESPA}).
 

\begin{thebibliography}{9}

\bibitem{AbadieBussFerraro}
{\sc F.~Abadie, A.~Buss, and D.~Ferraro}, {\em Amenability and approximation properties for partial actions and {F}ell bundles}, Bull. Braz. Math. Soc. (N.S.) 53 (2022), 173--227.

\bibitem{AD_systemes_1987}
{\sc C.~Anantharaman-Delaroche}, {\em Syst\`emes dynamiques non commutatifs et moyennabilit\'{e}}, Math. Ann. 279 (1987), 297--315.


\bibitem{AntPerThi_tensor_2018}
{\sc R.~Antoine, F.~Perera, and H.~Thiel}, {\em Tensor products and regularity properties of {C}untz semigroups}, Mem. Amer. Math. Soc. 251 (2018), viii+191.

\bibitem{BarSza_cartanII_2020}
{\sc S.~Barlak, and G.~Szab\'{o}}, {\em Rokhlin actions of finite groups on {UHF}-absorbing {${C}^*$}-algebras}, Trans. Amer. Math. Soc. 369 (2017), 833--859.


\bibitem{BeardenCrann}
{\sc A.~Bearden, and J.~Crann}, {\em Amenable dynamical systems over locally compact groups}, Ergodic Theory Dynam. Systems 42 (2022), 2468--2508.

\bibitem{Black_K}
{\sc B.~Blackadar}, {\em {$K$}-theory for operator algebras}, Mathematical Sciences Research Institute Publications 5, Cambridge University Press, Cambridge, second edition (1988).

\bibitem{Black_06}
{\sc B.~Blackadar}, {\em Theory of $C^*$-algebras and von Neumann algebras, Operator Algebras and Non-commutative Geometry, III}, Encyclopaedia of Mathematical Sciences 122, Springer-Verlag, Berlin (2006).


\bibitem{BrownOzawa}
{\sc N.~P.~Brown, and N.~Ozawa}, {\em {$C^*$}-algebras and finite-dimensional approximations}, Graduate Studies in Mathematics 88, American Mathematical Society, Providence, RI (2008).



\bibitem{BusEchWil20}
{\sc A.~Buss, S.~Echterhoff, and R.~Willett}, {\em Injectivity, crossed products, and amenable group actions}, {$K$}-theory in algebra, analysis and topology 749, Contemp. Math., Amer. Math. Soc., [Providence], RI (2020), 105--137.

\bibitem{BusEchWil22}
{\sc A.~Buss, S.~Echterhoff, and R.~Willett}, {\em Amenability and weak containment for actions of locally compact groups on {$C^*$}-algebras}, arXiv:2003.03469 (2020), to appear in Mem. Amer. Math. Soc. 


\bibitem{new_classification}
{\sc J.~R.~Carri\'{o}n, J.~Gabe, C.~Schafhauser, A.~Tikuisis, and S.~White}, {\em Classifying $^*$-homomorphisms I: Unital simple nuclear $C^*$-algebras}, arXiv:2307.06480 (2023).

\bibitem{Cun_1981}
{\sc J.~Cuntz}, {\em {$K$}-theory for certain {$C^*$}-algebras}, Ann. of Math. 113 (1981), 181--197.

\bibitem{EGLN_classification_2016}
{\sc G.~A.~Elliott, G.~Gong, H.~Lin, and Z.~Niu}, {\em On the classification of simple amenable {$C^*$}-algebras with finite decomposition rank, {II}}, arXiv:1507.03437 (2016).

\bibitem{Gabe_Szabo_2022}
{\sc J.~Gabe, and G.~Szab{\'o}}, {\em The dynamical {K}irchberg-{P}hillips theorem}, arXiv:2205.04933 (2022), to appear in Acta Math.


\bibitem{GGKN22}
{\sc E.~Gardella, S.~Geffen, J.~Kranz, and P.~Naryshkin}, {\em Classifiability of crossed products by nonamenable groups}, J. Reine Angew. Math. 797 (2023), 285--312.

\bibitem{ESPA}
{\sc E.~Gardella, S.~Geffen, P.~Naryshkin, and A.~Vaccaro}, {\em Dynamical comparison and {$\mathcal Z$}-stability for crossed products of simple {$C^*$}-algebras}, Adv. Math. 438 (2024), Paper No.~109471.

\bibitem{GHV_2022}
{\sc E.~Gardella, I.~Hirshberg, and A.~Vaccaro}, {\em Strongly outer actions of amenable groups on {$\mathcal{Z}$}-stable nuclear {$C^*$}-algebras}, J. Math. Pures Appl. 162 (2022), 76--123.

\bibitem{GarLup_group_2021}
{\sc E.~Gardella, and M.~Lupini}, {\em Group amenability and actions on {$\mathcal{Z}$}-stable {$C^*$}-algebras}, Adv. Math. 389 (2021), Paper No.~107931, 33.

\bibitem{GP}
{\sc E. Gardella, and F. Perera}, {\em The modern theory of Cuntz semigroups of {$C^*$}-algebras}, arxiv:2212.02290 (2022).

\bibitem{GLN_classification_2015}
{\sc G.~Gong, H.~Lin, and Z.~Niu}, {\em Classification of finite simple amenable {$\mathcal Z$}-stable {$C^*$}-algebras}, arXiv:1501.00135 (2015).


\bibitem{Haa_quasitraces_2014}
{\sc U.~Haagerup}, {\em Quasitraces on exact {$C^*$}-algebras are traces}, C. R. Math. Acad. Sci. Soc. R. Can. 36 (2014), 67--92.

\bibitem{HigsonKasparov}
{\sc N.~Higson, and G.~Kasparov}, {\em {$E$}-theory and {$KK$}-theory for groups which act properly and isometrically on {H}ilbert space}, Invent. Math. 144 (2001), 23--74.

\bibitem{Ker_dimension_2020}
{\sc D.~Kerr}, {\em Dimension, comparison, and almost finiteness}, J. Eur. Math. Soc. (JEMS) 22 (2020), 3697--3745.

\bibitem{KirRor_nonsimple_2000}
{\sc E.~Kirchberg, and M.~R{\o}rdam}, {\em Non-simple purely infinite {$C^*$}-algebras}, Amer. J. Math. 122 (2000), 637--666.


\bibitem{Kishimoto}
{\sc A.~Kishimoto}, {\em Outer automorphisms and reduced crossed products of simple {$C^*$}-algebras}, Comm. Math. Phys. 81 (1981), 429--435.

\bibitem{Lance}
{\sc E.~C.~Lance}, {\em Hilbert {$C^*$}-modules}, London Mathematical
Society Lecture Note Series 210, Cambridge University Press, Cambridge (1995).

\bibitem{LinSR1}
{\sc H.~Lin}, {\em Hereditary uniform property {$\Gamma$}}, 	arXiv:2202.05720 (2022).

\bibitem{Ma_comparison_2019}
{\sc X.~Ma}, {\em Comparison and pure infiniteness of crossed products}, Trans. Amer. Math. Soc. 372 (2019), 7497--7520.


\bibitem{matuisato:1}
{\sc H.~Matui, and Y.~Sato}, {\em {$\mathcal{Z}$}-stability of crossed products by strongly outer actions}, Comm. Math. Phys. 314 (2012), 193--228.

\bibitem{matuisato:2}
{\sc H.~Matui, and Y.~Sato}, {\em {$\mathcal Z$}-stability of crossed products by strongly outer actions {II}}, Amer. J. Math. 136 (2014), 1441--1496.

\bibitem{MeyerNest}
{\sc R.~Meyer, and R.~Nest}, {\em The {B}aum-{C}onnes conjecture via localisation of categories}, Topology 45 (2006), 209--259.

\bibitem{Ozawa_Dixmier_2013}
{\sc N.~Ozawa}, {\em Dixmier approximation and symmetric amenability for {$C^*$}-algebras}, J. Math. Sci. Univ. Tokyo 20 (2013), 349--374.


\bibitem{Ozawa_Suzuki_2020}
{\sc N.~Ozawa, and Y.~Suzuki}, {\em On characterizations of amenable {$C^*$}-dynamical systems and new examples}, Selecta Math. (N.S.) 27 (2021), 1--29.


\bibitem{Phi_classification_2000}
{\sc N.~C.~Phillips}, {\em A classification theorem for nuclear purely infinite simple {$C^*$}-algebras}, Doc. Math. 5 (2000), 49--114.

\bibitem{Ror_1992}
{\sc M.~R{\o}rdam}, {\em On the structure of simple {$C^*$}-algebras tensored with a {UHF}-algebra. {II}}, J. Funct. Anal. 107 (1992), 255--269.

\bibitem{Ror_classification_2002}
{\sc M.~R{\o}rdam}, {\em Classification of nuclear, simple {$C^*$}-algebras}, Classification of nuclear {$C^*$}-algebras. {E}ntropy in operator algebras, Encyclopaedia Math. Sci. 126, Springer, Berlin (2002), 1--145.

\bibitem{Ror_2004}
{\sc M.~R{\o}rdam}, {\em The stable and the real rank of {$\mathcal Z$}-absorbing {$C^*$}-algebras}, Internat. J. Math. 15 (2004), 1065--1084.

\bibitem{RordamSierakowski}
{\sc M.~R{\o}rdam, and A.~Sierakowski}, {\em Purely infinite {$C^*$}-algebras arising from crossed products}, Ergodic Theory Dynam. Systems 32 (2012), 273--293.

\bibitem{RosSch_kunneth_1987}
{\sc J.~Rosenberg, and C.~Schochet}, {\em The {K}\"{u}nneth theorem and the universal coefficient theorem for {K}asparov's generalized {$K$}-functor}, Duke Math. J. 55 (1987), 431--474.


\bibitem{sato}
{\sc Y.~Sato}, {\em Actions of amenable groups and crossed products of {$\mathcal Z$}-absorbing {$C^*$}-algebras}, Operator algebras and mathematical physics, Adv. Stud. Pure Math. 80, Math. Soc. Japan, Tokyo (2019), 189--210.

\bibitem{Suz_Equivariant_2021}
{\sc Y.~Suzuki}, {\em Equivariant {$\mathcal O_2$}-absorption theorem for exact groups}, Compos. Math. 157 (2021), 1492--1506.

\bibitem{Suzuki_Every_2022}
{\sc Y.~Suzuki}, {\em Every countable group admits amenable actions on stably finite simple {$C^*$}-algebras}, arXiv:2204.04480 (2022), to appear in Amer. J. Math.

\bibitem{szabo:equivariantKP}
{\sc G.~Szab\'{o}}, {\em Equivariant {K}irchberg-{P}hillips-type absorption for amenable group actions}, Comm. Math. Phys. 361 (2018), 1115--1154.

\bibitem{Sza_Equivariant_2021}
{\sc G.~Szab\'{o}}, {\em Equivariant property ({SI}) revisited}, Anal. PDE 14 (2021), 1199--1232.

\bibitem{takesaki}
{\sc M.~Takesaki}, {\em Theory of operator algebras. {I}}, Encyclopaedia of Mathematical Sciences 125, Springer-Verlag, Berlin (2003).

\bibitem{ThielCuntzsemigrp}
{\sc H.~Thiel}, {\em The Cuntz semigroup}, Lecture notes. University of M\"unster, \url{http://hannesthiel.org/wp-content/OtherWriting/CuScript.pdf} (2016).

\bibitem{ThielSR1}
{\sc H.~Thiel}, {\em Ranks of operators in simple {$C^*$}-algebras with stable rank one}, Comm. Math. Phys. 377 (2020), 37--76.

\bibitem{WeggeOlsen}
{\sc N.~E.~Wegge-Olsen}, {\em {$K$}-theory and {$C^*$}-algebras}, Oxford University Press, New York (1993).


\bibitem{Win_structure_2018}
{\sc W.~Winter}, {\em Structure of nuclear {$C^*$}-algebras: from quasidiagonality to	classification and back again}, Proceedings of the {I}nternational {C}ongress of {M}athematicians---{R}io de {J}aneiro 2018. {V}ol. {III}. {I}nvited lectures, World Sci. Publ., Hackensack, NJ (2018), 1801--1823.


\bibitem{Wouters_22}
{\sc L.~Wouters}, {\em Equivariant $\mathcal{Z}$-stability for single automorphisms on simple {$C^*$}-algebras with tractable trace simplices}, Math. Z. 304 (2023), Paper No. 22, 36.

\end{thebibliography}
\end{document}